\documentclass[12pt]{amsart}

\usepackage{graphics,graphicx} 
\usepackage[T1]{fontenc}

\usepackage[cp1250]{inputenc}

\usepackage[arrow, matrix, curve]{xy}
\usepackage{amsmath,amsthm,amssymb}
\usepackage{amscd}

\newenvironment{proof of}[1]{\noindent\emph{Proof of #1.}}{$\qquad \square$\par}

\newcommand{\Prim}{\textrm{Prim}\,}

\DeclareMathOperator{\dashind}{-Ind}

\DeclareMathOperator{\im}{im}
\DeclareMathOperator{\coker}{coker}
\DeclareMathOperator{\Aut}{Aut}

\DeclareMathOperator{\spane}{span}
\DeclareMathOperator{\clsp}{\overline{span}}

\newcommand{\HH}{\mathcal H}

\newcommand{\K}{\mathcal K}

\newcommand{\X}{\widetilde X}
\newcommand{\XX}{\mathcal X}

\newcommand{\VV}{\mathcal V}

\newcommand{\LL}{\mathcal L}

\newcommand{\al}{\alpha}

\newcommand{\FF}{\mathcal F}

\newcommand{\OO}{\mathcal O}
\newcommand{\A}{ A}
\newcommand{\B}{\mathcal B}
\newcommand{\E}{\mathcal E}
\newcommand{\SA}{\widehat{ \A}}
\newcommand{\D}{\mathcal D}

\newcommand{\C}{\mathbb C}

\newcommand{\Z}{\mathbb Z}
\newcommand{\N}{\mathbb N}
\newcommand{\T}{\mathbb T}
\newcommand{\TT}{\mathcal T}

 \newtheorem{thm}{Theorem}[section]
 \newtheorem{cor}[thm]{Corollary}
 \newtheorem{lem}[thm]{Lemma}
 \newtheorem{prop}[thm]{Proposition}
 \theoremstyle{definition}
 \newtheorem{defn}[thm]{Definition}
 \theoremstyle{remark}
 \newtheorem{rem}[thm]{Remark}
 \newtheorem*{ex}{Example}
 \numberwithin{equation}{section}

\begin{document}
 \title[Interactions and graph algebras]{Crossed products  for interactions 
 \\ and graph algebras}

\author{B. K.  Kwa\'sniewski} 
 \email{bartoszk@math.uwb.edu.pl}
 \address{Institute of Mathematics, Polish Academy of Science,  ul. \'Sniadeckich 8, PL-00-956 Warszawa, Poland
 //
 Institute of Mathematics, University  of Bialystok,  ul. Akademicka 2, PL-15-267  Bialystok, Poland  
}
 

\begin{abstract}
We consider  Exel's interaction $(\VV,\HH)$ over a unital $C^*$-algebra $\A$, such that $\VV(A)$ and $\HH(A)$ are hereditary subalgebras of $A$. For the associated crossed product, we obtain a uniqueness theorem,  ideal lattice description, simplicity criterion and a version of Pimsner-Voiculescu exact sequence. These results cover the case of crossed products by endomorphisms with hereditary ranges and complemented kernels. As  model examples of  interactions not coming from endomorphisms we introduce and study in detail  interactions  arising from  finite graphs. 

The  interaction $(\VV,\HH)$ associated to a graph $E$ acts on the core $\FF_E$ of the graph algebra $C^*(E)$. By describing   a   partial homeomorphism of $\widehat{\FF_E}$ dual to $(\VV,\HH)$  we find the fundamental structure theorems for $C^*(E)$, such as Cuntz-Krieger uniqueness theorem,  as results concerning reversible  noncommutative dynamics on $\FF_E$. We also provide a  new approach to calculation of $K$-theory of $C^*(E)$ using only an induced partial automorphism of $K_0(\FF_E)$ and the six-term exact sequence.

\end{abstract}

   \thanks{This work was in part supported by Polish  National Science Centre  grant number  DEC-2011/01/D/ST1/04112. } 
   
\maketitle

\tableofcontents
\section{Introduction}  
In \cite{exel-inter} Exel  extended celebrated  Pimsner's construction \cite{p} of the (nowadays called) Cuntz-Pimsner algebras by introducing an intriguing  new concept of a generalized $C^*$-correspondence. The  leading example in \cite{exel-inter} arises from an \emph{interaction}  -  a  pair $(\VV,\HH)$ of positive linear maps on a $C^*$-algebra $A$ that  are mutual generalized inverses and such that the image of one map is in the multiplicative domain of the other. 
An interaction can  be considered  a `symmetrized'  generalization of a \emph{$C^*$-dynamical system}, i.e. a  pair $(\al,\LL)$ consisting of an endomorphism $\al:A\to A$ and its  transfer operator $\LL:A\to A$, \cite{exel2}. One can think of many  examples of   interactions naturally appearing in various problems, cf. \cite{exel3}, \cite{exel-renault}, \cite{exel4}. However, at present there is only one significant application of an interaction $(\VV,\HH)$ which is not  a $C^*$-dynamical system. Namely, in the recent paper \cite{exel4} Exel  showed that the  $C^*$-algebra $\OO_{n,m}$ introduced in \cite{Ara-Exel-Katsura}, 
is Morita equivalent to the crossed product $C^*(A,\VV,\HH)$ for an interaction $(\VV,\HH)$, over a commutative $C^*$-algebra $A$, where neither  $\VV$ nor $\HH$ is multiplicative.   Moreover,   for crossed products under consideration  general  structure theorems  known so far concern   only the case when the initial object is an injective endomorphism, cf. \cite{Paschke0}, \cite{Paschke1}, \cite{Szwajcar}, \cite{exel_vershik}, \cite{BRV}. In particular, there are no such theorems  for genuine  interactions, i.e when both $\VV$ and $\HH$ are not multiplicative.

The purpose  of the  present article is two fold.

 Firstly, we establish general tools to study the structure of  $C^*(A,\VV,\HH)$ for an accessible and, as the $C^*$-dynamical system case indicates, important class of interactions $(\VV,\HH)$. Thus this   might be a considerable step in understanding these new objects.  More precisely, crossed products associated with  $C^*$-dynamical systems $(\alpha, \LL)$  on a unital $C^*$-algebra $A$ boast their greatest successes in the case   $\alpha(A)$ is a hereditary subalgebra of $A$, cf. \cite{Ant-Bakht-Leb}, \cite{exel2}, \cite{Paschke0}, \cite{Paschke1},  \cite{Rordam}. Then $\LL$ is a \emph{corner retraction}, see  \cite[page 424]{Szwajcar}, \cite{kwa-exel}.  It is uniquely determined by  $\alpha$ and it is called a \emph{complete transfer operator} in \cite{Ant-Bakht-Leb}, see 
\cite{kwa-trans}, \cite{kwa-exel}. In the present paper we focus on  interactions $(\VV,\HH)$  for which both $\VV(A)$ and $\HH(A)$ are hereditary subalgebras of $A$. Then  $\VV(A)$ and $\HH(A)$  are automatically corners in $A$. We call such interactions \emph{corner interactions}. It turns out  that each mapping in such an interaction $(\VV,\HH)$ is completely determined by the other. This plus the obvious connotation to complete transfer operators make it tempting to call $(\VV,\HH)$ a complete interaction \cite{kwa-demon}, but we resist this temptation here. 
\\
We show  that for a corner interaction   $(\VV,\HH)$ the crossed product $C^*(A,\VV,\HH)$ defined in \cite{exel-inter} is the universal $C^*$-algebra generated by a copy of $A$ and  a partial isometry $s$ subject to relations 
 $$
\VV(a)=s(a)s^*, \qquad \HH(a)=s^*(a)s, \qquad a\in A.
$$
As a consequence $C^*(A,\VV,\HH)$ can  be   modeled as the crossed product $A\rtimes_X \Z$, \cite{aee}, of $A$ by a Hilbert bimodule  $X=AsA$. It also follows that $C^*(A,\VV,\HH)\cong C^*(A,\VV)\cong C^*(A,\HH)$ where $C^*(A,\VV)$ (resp. $C^*(A,\HH)$) is the crossed product of $A$   by the completely positive mapping $\VV$ (resp. $\HH$), as  introduced in \cite{kwa-exel}. We study $C^*(A,\VV,\HH)$ by applying general methods developed for Hilbert bimodules \cite{kwa} and $C^*$-correspondences \cite{katsura}. For instance, we have a naturally defined partial homeomorphism $\widehat{\VV}$ of $\SA$ dual to $(\VV,\HH)$. Identifying it with the inverse to the induced partial homeomorphism $X\dashind$ studied  in \cite{kwa} we obtain:   \emph{uniqueness theorem} - topological freeness of $\widehat{\VV}$ implies faithfulness of every representation of $C^*(A,\VV,\HH)$ which is faithful on $A$; ideal lattice description via $\widehat{\VV}$-invariant open sets when $\widehat{\VV}$ is free; and the simplicity    of $C^*(A,\VV,\HH)$ when $\widehat{\VV}$ is minimal and topologically free (see Theorem \ref{interactions} below).  Similarly, identifying  the abstract morphisms in Katsura's version of Pimsner-Voiculescu exact sequence \cite{katsura} we  get a natural cyclic exact sequence for $K$-groups of $C^*(A,\VV,\HH)$ (Theorem \ref{Voicu-Pimsner for interacts}). It  generalizes the corresponding exact  sequence obtained by Paschke  for injective  endomorphisms \cite{Paschke1}, which plays a crucial role, for instance, in \cite{Rordam}.

Secondly, we provide a detailed analysis of  nontrivial corner interactions  with  an interesting \emph{noncommutative dynamics} related to Markov shifts,  and graph $C^*$-algebras as crossed products. More specifically, already in \cite{Cuntz1977}  Cuntz considered  his $C^*$-algebras $\OO_n$ as  crossed products of the core UHF-algebras by  injective endomorphisms implemented by one of the generating isometries. As noticed by R\o rdam \cite[Example 2.5]{Rordam},  a similar reasoning can be performed for Cuntz-Krieger algebras $\OO_A$ by considering an isometry given by the sum of  all generating partial isometries with  properly restricted initial spaces. An analogous isometry in $\OO_A$, but in a sense canonically associated with the underlying dynamics of Markov shifts, was  found in \cite[proof of Theorem  4.3]{exel2}, cf. \cite[formula (4.18)]{Ant-Bakht-Leb}. For the graph $C^*$-algebra $C^*(E)$ associated with a row-finite  graph $E$ with no sources\footnote{We follow here the original conventions of \cite{kum-pask-rae}, \cite{bprs} and hence in the context of representations of graphs we consider  different orientation of edges than  in \cite{Raeburn}, \cite{BRV}, \cite{hr}} the corresponding isometry  appears implicitly in \cite[Theorem 5.1]{BRV} and  explicitly in \cite[Theorem 5.2]{hr}, see formula \eqref{partial isometry for graph} below. In particular, if we assume $E$ is finite, i.e. the sets of vertices and edges are finite, and $E$ has no sources, we know from \cite[Theorem 5.2]{hr} 
states that $C^*(E)$ is naturally isomorphic to the crossed product of the  AF-core $C^*$-algebra $\FF_E$ by an injective endomorphism with hereditary range implemented by the aforementioned isometry $s$. Thus we  have
$$
C^*(E)=C^*(\FF_E\cup\{s\}),\qquad s\FF_E s^*\subset \FF_E, \qquad s^*\FF_E s\subset \FF_E.
$$
Moreover, one can  notice that the above picture remains valid  for arbitrary finite graphs, possibly with sources.  The only difference is that  $s$ may be no longer an isometry but  a partial isometry. Hence the mapping $\FF_E \ni a \to sas^* \in \FF_E$ may be  no longer multiplicative (at least not on its whole domain) and then a natural framework for $C^*(E)$ is the crossed product for an interaction $(\VV,\HH)$ over $\FF_E$ where $\VV(\cdot):=s(\cdot)s^*$, $\HH(\cdot):=s^*(\cdot)s$. 
We call the pair $(\VV,\HH)$ arising in this way  a\emph{ graph interaction}. It  can be viewed from many different perspectives  as a model example illustrating and giving new insight, for instance, to the following   objects and issues that we hope to be pursued in the  future. 

$\bullet$ \emph{Interactions with nontrivial algebras and not multiplicative dynamics}. The crossed product $C^*(\FF_E,$ $\VV,\HH)$ is naturally isomorphic to  the graph $C^*$-algebra $C^*(E)$ (Proposition \ref{thm for Huef and Raeburn}). In general, $(\VV,\HH)$ is not a $C^*$-dynamical system and is not a part of a group interaction \cite{exel3}. We precisely identify  the values of $n\in \N$ for which $(\VV^n, \HH^n)$ is  an interaction (see Proposition \ref{powers of partial isometry}), and it turns out that such $n$'s  might have almost arbitrary distribution. Moreover, $(\VV^n, \HH^n)$ is an interaction for all $n\in \N$ if and only if $(\VV,\HH)$ is a $C^*$-dynamical system (which may happen even if $E$ has sources).

$\bullet$  \emph{Noncommutative Markov shifts}. The main motivation in \cite{exel2} for introducing $C^*$-dynamical systems $(\alpha,\LL)$  was to realize  Cuntz-Krieger algebras $\OO_A$ as crossed products of the underlying Markov shifts, which was in turn suggested by \cite[Proposition 2.17]{CK}. In terms of graph $C^*$-algebras the relevant statement, see \cite[Theorem 5.1]{BRV},  says that when $E$ is finite and has no sinks, then  $C^*(E)$ is isomorphic to Exel crossed product $\D_E\rtimes_{\phi_E,\LL} \N$ where $\D_E\cong C(E^\infty)$  is a canonical masa in $\FF_E$. The spectrum of $\D_E$ is identified with the space of infinite paths $E^\infty$, $\phi_E$ is a transpose to the Markov shift on $E^\infty$ and $\LL$ is its classical Ruelle-Perron-Frobenious operator. Both $\phi_E$  and $\LL$ extend  naturally  to completely positive maps on $C^*(E)$ and the extension of $\phi_E$ is  called the noncommutative Markov shift, cf. e.g. \cite{jp}. However, from the point of view of the crossed product construction  the predominant role is played by  $\LL$, see \cite{kwa-exel}. In particular, $\LL=\HH$  where $(\VV,\HH)$ is the graph  interaction, and $\FF_E$ is a minimal $C^*$-algebra invariant under $\VV$ and containing $\D_E$. Thus  there are good reasons  to regard the graph interaction $(\VV,\HH)$ as an alternative candidate for the noncommutative counterpart of the Markov shift. Our dual and $K$-theoretic pictures of $(\VV,\HH)$   (see Theorem \ref{shifts of path representations} and Proposition \ref{label for K_0 automorphism}, respectively)   support  this point of view. 

$\bullet$  \emph{Graph $C^*$-algebras}. The structure of graph algebras was originally studied via grupoids \cite{kum-pask-rae}, \cite{kum-pask-rae-ren}, and $K$-theory was calculated using a  dual Pimsner-Voiculescu exact sequence and skew products of initial graphs \cite{Raeburn Szymanski}, \cite{Raeburn}. The corresponding results can also be achieved in the realm of partial actions of free groups on certain commutative $C^*$-algebras, see \cite{EL1}, \cite{EL2}.  We present here another approach, based on interactions.  We show that the partial homeomorphism  $\widehat{\VV}$ dual to $\VV$ is topologically free if and only if $E$ satisfies the so-called condition $(L)$ \cite{bprs}. Hence we derive the Cuntz-Krieger uniqueness theorem  \cite{kum-pask-rae-ren},  \cite{bprs}, \cite{Raeburn} from our general uniqueness theorem for interactions. Similarly, we  see that freeness of $\widehat{\VV}$  is equivalent to  condition (K) for $E$ \cite{kum-pask-rae-ren},  \cite{bprs}. Thus minimality and freeness of $\widehat{\VV}$ is equivalent to the known simplicity criteria for $C^*(E)$.  Moreover, it turns out  that pure infiniteness of $C^*(E)$, as defined in \cite{Lac-Spiel}, \cite{kum-pask-rae},  is equivalent to a very strong version of topological freeness of $\widehat{\VV}$ (see Remark \ref{pure infiniteness remark}), which therefore might be considered an instance of a noncommutative version of local boundary action, see \cite{Lac-Spiel}. Finally, our approach to calculation of $K$-groups for $C^*(E)$ seems to be the most direct upon the existing ones; it uses only direct limit description of the AF-core $\FF_E$ and the cyclic six-term exact sequence. 

$\bullet$  \emph{Topological freeness}. The condition known as topological freeness was for the first time explicitly stated in \cite{Donovan} where the author use it to show, what we call here, uniqueness theorem. Namely, he proved that topological freeness of a homeomorphism dual to an automorphisms $\alpha$ of a $C^*$-algebra $A$ implies that any representation of $A\rtimes_\alpha \Z$ whose restriction to $A$ is injective, is automatically faithful.  The converse implication (equivalence between topological freeness and the aforementioned uniqueness property) in the case $A$ is noncommutative turned out to be a difficult problem. It was proved in \cite[Theorem 10.4]{OlesenPedersen3} combined with \cite[Theorem 2.5]{OlesenPedersen2}, see also \cite[Remark 4.8]{OlesenPedersen2}, under the assumption that $A$ is separable. The proof is nontrivial and passes  through conditions involving such notions as Connes spectrum, inner derivations, or proper outerness. Since it is known that condition (L) is necessary for Cuntz-Krieger uniqueness theorem to hold,  our explicit characterization of topological freeness for graph interactions (see Theorem \ref{dynamical dychotomy}) serves as a good illustration and a starting point for further generalizations of the aforementioned notions and facts. 

$\bullet$  \emph{Dilations of completely positive maps}. Let us consider a $C^*$-algebra $C^*(A\cup \{s\})$ generated by a  $C^*$-algebra $A$ and  a partial isometry $s$ such that $sAs^* \subset A$. Also assume that $A$ and $C^*(A\cup \{s\})$ have a common unit. Then $\VV(\cdot)=s(\cdot)s^*$ is a completely positive map on $A$ sending the unit to an idempotent (this is a general form of such mappings, cf. \cite{kwa-exel}). We may put $\HH(\cdot):=s^*(\cdot) s$ and then one can see that 
$$
B:=\clsp\left\{a_0 s^* a_{1} s^*a_2 ... s^*a_{n} s b_1  s b_2 ... s b_n: a_i, b_i\in A, n\in \N \right\}
$$
is the smallest $C^*$-algebra  preserved by $\HH$ and containing $A$.  Plainly, the pair $(\VV, \HH)$ is a corner interaction on $B$. Hence, potentially,  our  results could  be applied to study the structure of $C^*(A\cup\{s\})=C^*(B\cup\{s\})$. Nevertheless, the  dilation of $\VV$  from $A$ to $B$ is a  nontrivial procedure and in general depends on the initial representation of $\VV$ via $s$. The core algebras $B$ arising in this way are studied in detail for instance in \cite{KajiWata}, \cite{kwa-logist}, \cite{kwa-ext}.  
Our analysis of the graph interaction $(\VV,\HH)$  can be viewed as a case study of the above situation when  $A\cong \C^N$ is a finite dimensional commutative $C^*$-algebra, see Remark \ref{dilation remark}. In particular, Theorem \ref{shifts of path representations} can be interpreted as that the partial homeomorphism  dual to  a dilation  of the Ruelle-Perron-Frobenius operator $\HH=\LL$  (from  $A$ to $B=\FF_E$) is a quotient of the Markov shift.

\medskip

We begin by presenting relevant notions and statements concerning Hilbert bimodules and briefly clarifying their relationship with generalized $C^*$-correspondences. General corner interactions are studied in Section \ref{interactions crossed products}. Section \ref{Cuntz-Krieger algebras section} is devoted to analysis of graph interactions.  
\subsection{Preliminaries on Hilbert bimodules} 

Throughout  $\A$ is a $C^*$-algebra which (starting from Section \ref{interactions crossed products}) will always be  unital. By  homomorphisms, epimorphisms, etc. between $C^*$-algebras we always mean  $*$-preserving maps.  All  ideals in $C^*$-algebras are  assumed to be closed and two sided.  We adhere to the convention that $$\beta(A,B)=\clsp\{\beta(a,b)\in C: a\in A,b\in B\}$$ 
for  maps $\beta\colon A\times B\to C$  such as inner products, multiplications or representations.  

As in \cite{kwa} 
we say that a partial homeomorphism $\varphi$ of a topological space $M$, i.e. a homeomorphism whose domain $\Delta$ and range $\varphi(\Delta)$ are open subsets of $M$,   is   {\em
topologically free} if  for any $n>0$ the set of fixed points for $\varphi^n$ (on its natural domain) has empty interior.  A set $V$ is $\varphi$-\emph{invariant}    if 
$
\varphi(V\cap \Delta)= V\cap \varphi(\Delta). 
$
If there are no nontrivial closed invariant sets, then   $\varphi$   is  called  {\em
minimal}, and  $\varphi$  is said to be (residually) \emph{free}, if it is topologically free on every closed invariant set (in the  Hausdorff space case   this amounts to requiring that $\varphi$ has no periodic points).

Following  \cite[1.8]{BMS} and \cite{aee} by a \emph{Hilbert bimodule over $\A$} we mean  $X$ which is both a left Hilbert  $\A$-module and a right Hilbert  $\A$-module  with respective inner products  $\langle \cdot,\cdot \rangle_{\A}$ and  ${_{\A}\langle} \cdot,\cdot \rangle$ satisfying the so-called \emph{imprimitivity condition}:
$
 x \cdot \langle y ,z \rangle_\A = {_\A\langle} x , y  \rangle \cdot z$, for all $x,y,z\in X$.   A \emph{covariant representation} of  $X$ is a pair $(\pi_\A,\pi_{X})$  consisting of a homomorphism $\pi_\A:A\to \B(H)$ and a linear map  $\pi_X:X\to \B(H)$  such that  
 \begin{equation}\label{covariant part one}
 \pi_X(ax)=\pi_\A(a)\pi_X(x),\qquad \pi_X(xa)=\pi_X(x)\pi_\A(a),
 \end{equation}
 \begin{equation}\label{covariant part two}
 \pi_\A(\langle x ,y \rangle_\A)=\pi_X(x)^*\pi_X(y),\qquad  \pi_\A( {_\A\langle} x , y  \rangle)=\pi_X(x)\pi_X(y)^*,
 \end{equation}
 for all $a\in \A$, $x,y \in X$. The \emph{crossed product} $\A\rtimes_X \Z$ is a  $C^*$-algebra generated by a copy of $A$ and $X$   universal with respect to covariant  representations of $X$, see \cite{aee}. It is equipped with the \emph{circle gauge action} $\gamma=\{\gamma_z\}_{z\in \T}$ given on generators by $\gamma_z(a)=a$ and $\gamma_z(x)=zx$, for $a\in A$, $x\in X$, $z\in \T=\{z\in \C:|z|=1\}$.
 
As it is standard, we abuse the language and denote by $\pi$ both an irreducible representation of $A$ and its equivalence class in the spectrum $\SA$ of $A$. It should not cause confusion when we consider induced representations, as for a Hilbert bimodule $X$ over $A$ the induced representation functor  $X\dashind$   preserves such  classes. We briefly recall, and refer to \cite{morita} for all necessary details, that $X\dashind$ maps  a representation $\pi:\A\to \B(H)$  to a representation $X\dashind(\pi):A\to \B(X\otimes_\pi H)$ where the Hilbert  space $X\otimes_\pi H$ is generated by simple tensors $x\otimes_\pi h$, $x\in X$, $h\in H$, satisfying
$
\langle x_1\otimes_\pi h_1, x_2\otimes_\pi h_2 \rangle = \langle h_1,\pi(\langle x_1, x_2 \rangle_{\A})h_2\rangle,
$
and
$$
X\dashind(\pi)(a)  (x\otimes_\pi h) = (a x)\otimes_\pi h, \qquad a\in A.
$$
The spaces $\langle X ,X \rangle_\A$ and ${_\A\langle} X , X  \rangle$ are ideals in $A$ and the bimodule $X$ implements  a Morita equivalence between them. Hence    $X\dashind: \widehat{\langle X ,X \rangle_\A} \to \widehat{{_\A\langle} X , X  \rangle}$ is a homeomorphism which we may naturally treat as  a partial homeomorphism of $\SA$, see \cite{kwa}. 

The results of \cite{kwa} can be summarized as follows.
\begin{thm}\label{main result} Let $X\dashind$ be a partial homeomorphism of $\SA$, as  described above.
\begin{itemize}
\item[i)]
If $X\dashind$ is  topologically
free, then   every faithful  covariant representation  $(\pi_\A,\pi_X)$   of  $X$ `integrates' to
the faithful representation of $A\rtimes_X \Z$.
\item[ii)] If  $X\dashind$ is free, then $J \mapsto  \widehat{J \cap \A}$ is a lattice isomorphism
 between ideals in $\A\rtimes_X \Z$ and  open invariant sets in $\SA$.
 \item[iii)] If  $X\dashind$ is  topologically
free and minimal, then $\A\rtimes_X \Z$ is simple.
\end{itemize}
\end{thm}
\begin{rem}\label{remark about domains}
The map $X\dashind$ is a lift of the so-called Rieffel homeomorphism  $h_X:\Prim \langle X ,X \rangle_\A \to \Prim {_A\langle} X , X  \rangle$,  cf.  \cite[Corollary 3.33]{morita}, \cite[Remark 2.3]{kwa}. Plainly, topological freeness of $(\Prim(\A), h_X)$ implies  topological freeness of $(\SA, X\dashind)$, but the converse is not true and as we will see, cf. Example \ref{remark on topological freeness}  below, Cuntz algebras $\OO_n$  provide an excellent example of this phenomenon. 
\end{rem}
\begin{rem}\label{remark on szwajcar}
In \cite{Szwajcar} Schweizer showed that if $X$ is a full nondegenerate $C^*$-correspondence over a unital $C^*$-algebra $A$, then the Cuntz-Pimsner algebra  $\OO_X$, defined as in \cite{p}, is simple if and only if $X$ is minimal and aperiodic \cite[Definition 3.7]{Szwajcar}.  Clearly, if $X$ is a Hilbert bimodule, minimality of $X\dashind$ is equivalent to the minimality of $X$ and topological freeness of $X\dashind$ implies the aperiodicity of $X$. Moreover, the algebras $\OO_X$ and $\A\rtimes_X \Z$  coincide if and only if ${_A\langle} X, X \rangle$ is an essential ideal in $A$ (which in turn is equivalent to  injectivity of the left action of $A$ on $X$). In particular, if the ideal  ${_A\langle} X, X \rangle$ is essential in $A$ and $\langle X, X\rangle_A=A$ is unital, then \cite[Theorem 3.9]{Szwajcar} implies that $\A\rtimes_X \Z$ is simple iff $X$ is minimal and aperiodic.
\end{rem}
Let us fix a Hilbert bimodule $X$  over $A$. We  notice that it is naturally equipped with the ternary ring operation 
$$
[x,y,z]:=x \langle y,z\rangle_{\A}= {_{\A}\langle}x,y\rangle z, \qquad x,y,z \in X,
$$
making it into a  generalized correspondence over $A$, as defined in \cite[Definition 7.1]{exel-inter}. Alternatively, this generalized correspondence could be   described in terms of \cite[Proposition 7.6]{exel-inter} as the triple $(X,\lambda, \rho)$ where we consider  $X$ as a ${_A\langle} X , X  \rangle$-$\langle X ,X \rangle_A$-Hilbert bimodule  and define homomorphisms $\lambda:A\to {_\A\langle} X , X  \rangle$ and $\rho:A \to \langle X ,X \rangle_\A$  to be  (necessarily unique) extensions of the identity maps. 

The following fact should be compared with \cite[Proposition 7.13]{exel-inter}.
\begin{prop}\label{normal crossed product as a generalized one}
The crossed product  $\A\rtimes_X \Z$ of the Hilbert bimodule $X$ is naturally isomorphic to the covariance algebra $C^*(A,X)$, as defined in \cite[7.12]{exel-inter}, for  $X$ treated as a  generalized correspondence. 
\end{prop}
\begin{proof}
The Toeplitz algebra $\TT(\A,X)$ for the generalized correspondence $X$, see \cite[page 57]{exel-inter}, is a universal $C^*$-algebra generated by a copy of $\A$ and $X$ subject to all $\A$-$\A$-bimodule relations plus the ternary ring  relations:
\begin{equation}\label{teranry relation}
x y^* z = x \langle y,z\rangle_{\A}= {_{\A}\langle}x,y\rangle z, \qquad x,y,z \in X.
\end{equation}
The $C^*$-algebra $C^*(\A,X)$ is  the quotient $\TT(\A,X)/(J_\ell+J_r)$ where  $J_\ell$ (respectively $J_r$) is an ideal in $\TT(\A,X)$ generated by the elements $a-k$ such that $a\in(\ker\lambda)^\bot$,  $k \in XX^*$ (resp.  $a\in (\ker\rho)^\bot$,  $k \in X^*X$) and  
\begin{equation}\label{redundancy relations}
ax =kx \quad (\textrm{or resp. } xa=xk) \quad \textrm{ for all } x\in X.
\end{equation}
Note that $(\ker\lambda)^\bot={_A\langle} X, X \rangle$ and $(\ker\rho)^\bot=\langle X, X \rangle_A$. By \eqref{teranry relation},   $XX^*$ and $X^*X$ are  $C^*$-subalgebras of $\TT(\A,X)$. Hence using approximate units argument we see that when $a$ is  fixed  relations  \eqref{redundancy relations} determine $k$ uniquely. It follows  that  
$$
J_\ell=\clsp\{{_A\langle} x,y\rangle -x y^*:x,y \in X\},\qquad J_r=\clsp\{ \langle x,y\rangle_{\A} -x^* y : x,y \in X \},
$$ 
because if (for instance) $a-k\in J_\ell$ where $a=\sum_{i=1}^n {_A\langle} x_i, y_i \rangle\in (\ker\lambda)^\bot$ and $k\in X^*X$, then by  \eqref{teranry relation},  $ax=\sum_{i=1}^n x_i y_i^* x $ for all $x\in X$ and thus $k=\sum_{i=1}^n x_i y_i^*$. 
 
 Accordingly, both 
 $C^*(A,X)$ and $\A\rtimes_X \Z$ are universal $C^*$-algebras generated by copies of $A$ and $X$ subject to the same relations. 
\end{proof}

  Katsura obtained in \cite{katsura}   a version of  the Pimsner-Voiculescu exact sequence for general $C^*$-correspondences and their $C^*$-algebras. We recall it in the case $X$ is a Hilbert bimodule and in a form suitable for our purposes.  We consider the linking algebra $D_X=\K(X\oplus A)$  in the following matrix representation
$$
D_X=\left(\begin{matrix} 
 \K(X) & X
\\
\widetilde{X} & A
\end{matrix}\right),
$$
where $\widetilde{X}$ is the dual Hilbert bimodule of $X$, cf. e.g. \cite[pages 49, 50]{morita}, and let $\iota:{_A\langle X,X\rangle} \to A$, $\iota_{11}: \K(X)\to D_X$ and $\iota_{22}: A\to D_X$  be  inclusion maps;   $\iota_{11}(a)=\left(\begin{matrix} 
a& 0
\\
0 & 0
\end{matrix}\right)$, $\iota_{22}(a)=\left(\begin{matrix} 
0& 0
\\
0 & a
\end{matrix}\right)$. By \cite[Proposition B.3]{katsura}, $(\iota_{22})_*:K_*(A)\to K_*(D_X)$	 is an isomorphism and  by \cite[Theorem 8.6]{katsura}  the following  sequence is exact: 
\begin{equation}\label{Katsura Pimsner voiculescu sequence}
\begin{xy}
\xymatrix{
      K_0({_A\langle X,X\rangle}) \ar[rr]_{\,\,\,\,\,\,\,\,\,\,\iota_*-(X_*\circ\phi_*)} & & K_0(A) \ar[rr]_{(i_A)_*\,\,\,\,\,\,\,\,\,}  & &   \ar[d]  K_0(A\rtimes_X \Z )
             \\
   K_1(A\rtimes_X \Z) \ar[u]  & &  K_1(A)  \ar[ll]_{\,\,\,\,\,\,\,\,\,\,\,\,(i_A)_*}  &  & \ar[ll]_{\iota_*-(X_*\circ\phi_*)}  K_1({_A\langle X,X\rangle})
              } 
  \end{xy} 
\end{equation}
where $\phi: A\to \LL(X)$ is the homomorphism implementing left action of $A$ on $X$, and 
 $X_*:K_*({_A\langle X,X\rangle})\to K_*(A)$  is the composition of $(\iota_{11})_*:K_*({_A\langle X,X\rangle})\to K_*(D_X)$ and the inverse to the isomorphism $(\iota_{22})_*:K_*(A)\to K_*(D_X)$.

\section{Complete interactions and their crossed products}\label{interactions crossed products} 
In this section, following closely the relationship between $C^*$-dynamical systems and interactions, we introduce corner interactions,  describe
the structure of the associated crossed product and establish fundamental tools for its analysis (Theorems \ref{interactions}, \ref{Voicu-Pimsner for interacts}). 

\subsection{Interactions and $C^*$-dynamical systems} It is instructive to consider interactions as  generalization of pairs $(\al,\LL)$, sometimes called Exel systems \cite{hr}, consisting of an endomorphism $\alpha:\A\to \A$  and its \emph{transfer operator}, i.e.  a  positive linear  map $\LL:\A\to \A$  such that
$\LL(\alpha(a)b) =a\LL(b)$,  $a,b\in \A$, see \cite{exel2}. Then $\LL$ is  automatically continuous, $*$-preserving, and we also have: 
$
\LL(b\alpha(a)) =\LL(b)a$, $a,b\in \A$.
We say that a transfer operator $\LL$ is 
 \emph{regular} if $\al(\LL(1))=\al(1)$, or equivalently \cite[Proposition 2.3]{exel2}, if 
$\E(a):=\alpha (\LL(a))
 $
 is a conditional expectation from $\A$ onto $\al(\A)$. We note that originally \cite{exel2} Exel called such transfer operators non-degenerate. However, the use of the latter term is a bit unfortunate. For instance, it is used in the related context to mean a different property in \cite[page 60]{exel-inter}, and also there are historical reasons to change this name, see \cite{kwa-exel}.  
 
 It is important, see \cite{kwa-trans}, that the range of a regular transfer operator $\LL$ coincides with the annihilator $(\ker\al)^\bot$ of the kernel of $\al$ and $\LL(1)$ is the unit in $\LL(\A)=(\ker\al)^{\bot}$, so in particular the latter is a complemented ideal. 
 \begin{defn}
 A pair $(\al,\LL)$ where $\LL:\A\to \A$ is a regular transfer operator for an endomorphism $\al:\A\to\A$ will be called a \emph{$C^*$-dynamical system}.
 \end{defn}
 A dissatisfaction concerning asymmetry in the $C^*$-dynamical system $(\al,\LL)$;  $\al$ is multiplicative while $\LL$ is `merely' positive linear, lead the author of \cite{exel-inter} to the following more general notion.
\begin{defn}[\cite{exel-inter}, Definition 3.1]\label{interaction definition}  The pair  $(\VV,\HH)$  of  positive   linear maps $\VV,\HH:\A\to\A$ is called an \emph{interaction} over $\A$  if 
  \begin{itemize}
     \item[(i)] $\VV\circ \HH\circ\VV=\VV$,
  \item[(ii)] $\HH\circ \VV\circ \HH=\HH$,
  \item[(iii)] $\VV(ab)=\VV(a)\VV(b)$, if either $a$ or $b$ belong to $\HH(\A)$,
    \item[(iv)] $\HH(ab)=\HH(a)\HH(b)$, if either $a$ or $b$ belong to $\VV(\A)$.
 \end{itemize}
\end{defn}
\begin{rem} An interaction $(\VV,\HH)$,  or even a  $C^*$-dynamical system $(\al,\LL)$, in general does not generate a semigroup of interactions \cite{kwa-demon} and all the more is not an element of a group interaction in the sense of \cite{exel3}. This will be a generic case in our example arising from graphs, cf.  Proposition \ref{powers of partial isometry} below. Accordingly, in general the facts proved  in \cite{kwa-demon}, \cite{exel3}, can not be applied in our present context.  
\end{rem}

Let $(\VV,\HH)$ be an interaction. By \cite[Propositions 2.6, 2.7]{exel-inter}, $\VV(\A)$ and $\HH(\A)$ are  $C^*$-subalgebras of $\A$,
   $\E_{\VV}:=\VV\circ \HH$ is a conditional expectation onto $\VV(\A)$,
    $\E_{\HH}:=\HH\circ\VV$ is a conditional expectation onto $\HH(\A)$,
 and the mappings
  $$
   \VV: \HH(\A) \to \VV(\A),\qquad  \HH: \VV(\A) \to \HH(\A)
  $$
   are isomorphisms, each being the inverse of the other. Actually we have   
\begin{prop}\label{prop stanislaw} The relations $\E_{\VV}=\VV\circ \HH$,  $\E_{\HH}=\HH\circ\VV$, $\theta=\VV|_{\E_{\HH}(A)}$ yield a one-to-one correspondence between  interactions $(\VV,\HH)$ and triples $(\theta,\E_{\VV}, \E_{\HH})$ consisting of two conditional expectations $\E_{\VV}, \E_{\HH}$ and an isomorphism $\theta: \E_{\HH}(A) \to \E_{\VV}(A)$.
   \end{prop}
   \begin{proof}
It suffices to verify that  if  $(\theta,\E_{\VV}, \E_{\HH})$ is  as in the assertion, then  
$\VV(a):=\theta(\E_{\HH}(a))$ and  $\HH(a):=\theta^{-1}(\E_{\VV}(a))$ form an interaction. This is straightforward.
  \end{proof}
Recall that   the  $C^*$-algebra $A$ has the unit $1$. It follows that the algebras involved in an interaction are automatically also unital.
      \begin{lem}\label{lem stanislaw} If $(\VV,\HH)$ is an interaction, then $\VV(1)=\E_{\VV}(1)$ and $\HH(1)=\E_{\HH}(1)$ are  units  in $\VV(\A)$ and $\HH(\A)$, respectively (in particular, they are  projections).
   \end{lem}
   \begin{proof}
Let us observe  that
\begin{align*}
\E_{\VV}(1)&=\VV(\HH(1))=\VV(\HH(1)1)=\VV(\HH(1))\VV(1)=\VV(\HH(1))\VV\big(\HH(\VV(1))\big)
\\
&=\VV\big(\HH(1)\HH(\VV(1))\big)=\VV\big(\HH(1\VV(1))\big)=\VV\big(\HH(\VV(1))\big)=\VV(1).
\end{align*}
Therefore we have $
\VV(a)=\E_{\VV}(\VV(a))= \E_{\VV}(1 \VV(a))=\E_{\VV}(1) \VV(a)=  \VV(1)\VV(a)
$ for arbitrary $a\in A$. It follows that  $\VV(1)$ is the unit in $\VV(\A)$ and a similar argument works for $\HH$.
  \end{proof}
The following statement generalizes \cite[Proposition 3.4]{exel-inter}.
    \begin{prop}\label{from C--dynamical systems to interactions}
     Any $C^*$-dynamical system $(\al,\LL)$ is an interaction.
    \end{prop}
    \begin{proof}
Consider the conditions (i)-(iv) in Definition \ref{interaction definition}. Since  $\al\circ\LL\circ\al=\E\circ \al=\al$, (i)  is satisfied. To see (ii) recall  that $\LL(1)$ is the  unit in $\LL(A)$, cf. \cite[Proposition 1.5]{kwa-trans}, and therefore   
    $$
      \LL(\al(\LL(a)))=\LL(1 \al(\LL(a)))=\LL(1)\LL(a)=\LL(a)
    .$$
 Condition (iii) is trivial for $(\al,\LL)$, and   (iv) holds because 
 $$
 \LL(a\al(b))=\LL(a) b= \LL(a)  \LL(1)b= \LL(a) \LL(1 \al(b))  =\LL(a) \LL(\al(b)),
 $$
 and by passing to adjoints we also get $\LL(\al(b)a)=\LL(\al(b))\LL(a)$.
    \end{proof}
    
As shown in \cite{Ant-Bakht-Leb},  in the case  the conditional expectation $\E=\al\circ \LL$ is given by 
\begin{equation}\label{complete transfer operator condition}
\E(a) =\alpha(1)a\alpha(1),\qquad a\in\A,
\end{equation}
 there  is a very natural   crossed product   associated to the $C^*$-dynamical system $(\al,\LL)$.
This crossed product  coincides with the one introduced in   \cite{exel2}  and  is sufficient to cover many classic constructions, see \cite{Ant-Bakht-Leb}.

 A transfer operator for which \eqref{complete transfer operator condition} holds is called   \emph{complete}  \cite{Ant-Bakht-Leb}, \cite{kwa-trans}. It is a \emph{corner retraction} \cite{Szwajcar}, \cite{kwa-exel}. By \cite{kwa-trans}  a given endomorphism $\al$ admits a  complete transfer operator $\LL$  if  and only if  $\ker\al$ is a complemented ideal and $\al(\A)$ is a hereditary subalgebra of $\A$. In this case $\LL$ is a unique regular transfer operator for $\al$, see \cite{kwa-trans}, \cite{Ant-Bakht-Leb}, \cite{Szwajcar}, \cite{kwa-exel}. 
We naturally generalize  the aforementioned concepts to  interactions, cf. also \cite{kwa-demon}.
   
   \begin{defn}
  An interaction $(\VV,\HH)$ will be called a\emph{ corner interaction} if   $\VV(\A)$ and $\HH(\A)$ are hereditary  subalgebras of $\A$.
   \end{defn}   
\begin{prop}\label{characterization of corner interactions}
An interaction $(\VV,\HH)$ is corner if and only if   $\VV(\A)=\VV(1)\A\VV(1)$  and  $\HH(\A)=\HH(1)\A\HH(1)$ are corners in $\A$.
Moreover, for a corner interaction $(\VV,\HH)$ the following conditions are equivalent 
\begin{itemize} 
\item[i)]  $(\VV,\HH)$ is a  $C^*$-dynamical system  (with a complete transfer operator),
 \item[ii)]  $\VV$ is multiplicative,
 \item[iii)] $\ker \VV$ is an ideal in $A$,
 \item[iv)]   $\HH(A)$ is an ideal in $A$, 
 \item[v)] $\HH(1)$ lies in the center of $A$.
\end{itemize}
\end{prop}
\begin{proof}
For  the first part of  the assertion apply  Lemma \ref{lem stanislaw} and notice that if $B$ is a hereditary subalgebra of $\A$ and $B$ has a  unit $P$, then $B=P\A P$.   To show the second part of assertion let us suppose   that $(\VV,\HH)$ is a corner interaction. 

The implications  i) $\Rightarrow$ ii) $\Rightarrow$ iii) and the equivalence iv) $\Leftrightarrow$ v) are clear.

iii) $\Rightarrow$ v).   By the first part of the assertion  $\VV$ is isometric on $\HH(1)A\HH(1)$ and thus
$\ker \VV\cap \HH(1)A\HH(1)=\{0\}$.
 In view of  Lemma \ref{lem stanislaw}, for any $a\in A$ we have  $a(1-\HH(1))\in \ker\VV$. Hence  if $\ker\VV$ is an ideal,  then  $\HH(1)a\big(1-\HH(1)\big) a^*\HH(1)\in (\ker\VV) \cap \HH(1)A\HH(1)=\{0\}$, that is  $\HH(1)a(1-\HH(1))=0$ which means that $\HH(1)a=a\HH(1)$.

v) $\Rightarrow$ i).
By the first part of the assertion $\E_\HH(a)=\HH(1)a\HH(1)$. Thus, for any $a,b\in A$, we have
\begin{align*}
\VV(ab)&=\VV(\E_{\HH}(ab))= \VV(\HH(1)ab\HH(1))= \VV(a \HH(1) b\HH(1))
\\
&= \VV(a \E_{\HH}(b))=\VV(a)\VV(\E_{\HH}(b))=\VV(a)\VV(b).
\end{align*}
Hence $\VV$ is an endomorphism of $A$. The map $\HH$ is a transfer operator for $\VV$ because
$$
\HH(a\VV(b))=\HH(a) \HH(\VV(b)))=\HH(a) \HH(1)b\HH(1)=\HH(a) b.
$$
\end{proof}
As it is indicated by the uniqueness of the complete transfer operator, it turns out that  each mapping   in a corner interaction determines the other.

\begin{prop}\label{proposition about uniqueness of interactions}
  A positive linear map $\VV:\A\to \A$ is a part of a non-zero corner  interaction  $(\VV,\HH)$  if and only if 
 $\|\VV(1)\|=1$, $\VV(\A)$ is  a hereditary subalgebra of $\A$  and there is a  projection $P\in \A$ such that $\VV:P\A P \to \VV(\A)$ is an isomorphism.
 
Moreover, in the above equivalence   $P$  and $\HH$ are uniquely determined by $\VV$, and we have
\begin{equation}\label{definition of H}
\HH(a):=\VV^{-1}(\VV(1)a\VV(1)), \qquad a\in A,
\end{equation} 
where $\VV^{-1}$ is the inverse to $\VV:P\A P \to \VV(\A)$.
    \end{prop} 
    \begin{proof}
The necessity of the stated conditions follows from Proposition \ref{characterization of corner interactions} and Lemma \ref{lem stanislaw}. For the sufficiency note that  $\VV(P)$ is a unit in $\VV(A)$ and therefore $\VV(A)=\VV(P) A\VV(P)$, as $\VV(A)$ is  hereditary in $A$. In particular, $\E_\VV(a):=\VV(P)a\VV(P)$  is a conditional expectation onto $\VV(A)$.  We  define $\E_\HH(a):=\VV^{-1}(\VV(a))$ where $\VV^{-1}$ is the inverse to $\VV:P\A P \to \VV(\A)$. Then  $\E_{\HH}$  is an idempotent map of norm one because  $\|\E_{\HH}\|=\|\VV\|=\|\VV(1)\|=1$. Hence $\E_{\HH}$ is a conditional expectation onto $PAP$. By Proposition \ref{prop stanislaw}, the triple $(\VV,\E_{\VV},\E_{\HH})$ yields a (necessarily corner) interaction $(\VV,\HH)$ where $\HH(a)=\VV^{-1}(\VV(P)a\VV(P))$. In particular, it follows from Lemma \ref{lem stanislaw} that $\VV(P)=\VV(1)$, that is $\HH$ is given by \eqref{definition of H}.
    
What remains to be shown is  the uniqueness of  $P$. Suppose then that 
 $(\VV,\HH_i)$, $i=1,2$,  are  two corner  interactions and consider   projections $P_1:=\HH_1(1)$ and $P_2:=\HH_2(1)$. We have
        $$
    \VV(P_1P_2P_1)=\VV(P_2)=\VV(1)=\VV(P_1)=\VV(P_2P_1P_2),
      $$ 
      and as $\VV$ is injective on $\HH_i(\A)=P_i\A P_i$, $i=1,2$, it follows that $P_1P_2P_1=P_1$ and $P_2=P_2P_1 P_2$. This  implies $P_1=P_2$.  
    \end{proof} 
\subsection{Crossed product for corner interactions}\label{subsection two two} From now on  $(\VV,\HH)$ will always stand for a corner interaction. We define the corresponding  crossed product in universal terms.
  \begin{defn}
A \emph{covariant representation}  of $(\VV,\HH)$ is a pair $(\pi,S)$ consisting of a non-degenerate representation $\pi:\A\to \B(H)$ and an operator $S\in \B(H)$ (which is  necessarily a partial isometry) such that  
$$ 
S\pi(a)S^*=\pi(\VV(a)) \, \textrm{ and  } \, S^*\pi(a)S=\pi(\HH(a))\,\, \textrm{ for all } a\in \A.
$$
The \emph{crossed product for the  interaction} $(\VV,\HH)$ is the $C^*$-algebra  $C^*(A,\VV,\HH)$ generated by $i_\A(\A)$ and $s$ where $(i_\A,s)$ is a universal covariant representation of $(\VV,\HH)$. It is equipped with the \emph{circle gauge action} determined  by $\gamma_z(i_A(a))=i_A(a)$,  $a\in A$, and $\gamma_z(s)=zs$.
  \end{defn}
Obviously, the above definition generalizes the crossed product  studied in  \cite{Ant-Bakht-Leb}. In other words $C^*(A,\VV,\HH)$  coincides with Exel's crossed product \cite{exel2} when $(\VV,\HH)$ is a $C^*$-dynamical system. To show it is essentially  the same algebra as the one  associated to (general) interactions in \cite{exel-inter},  we  realize $C^*(A,\VV,\HH)$ as the crossed product for a Hilbert bimodule. 
  To this end, we conveniently adopt Exel's construction of his generalized correspondence associated to $(\VV,\HH)$, \cite[Section 5]{exel-inter}.

    Let $X_0=\A \odot \A$ be  the algebraic tensor product over the complexes, and let 
$
  \langle \cdot , \cdot  \rangle_\A$ and  ${_\A\langle} \cdot  , \cdot  \rangle$     
be the $\A$-valued sesqui-linear functions defined on $X_0\times X_0$ by
$$
 \langle a\odot b, c\odot d  \rangle_\A= b^* \HH(a^*c)d ,\qquad  {_\A\langle}  a\odot b, c\odot d  \rangle=a \VV(bd^*)c^*.
$$ 
We consider the linear space  $X_0$ as an $\A$-$\A$-bimodule with the natural module operations: $a \cdot (b\odot c)= ab \odot c$,  $(a \odot b)\cdot c= a\odot b  c$.
\begin{lem}\label{pre Hilbert bimodule construction}
A quotient of $X_0$ becomes naturally a pre-Hilbert $\A$-$\A$-bimodule. More precisely,
\begin{itemize}
\item[i)] the space $X_0$ with a function $\langle \cdot , \cdot  \rangle_\A$  (respectively  ${_\A\langle} \cdot  , \cdot  \rangle$) becomes a right (respectively left) semi-inner product $\A$-module;
\item[ii)]  the corresponding semi-norms 
$$
\|x\|_{\A}:=\|\langle x, x  \rangle_\A\|^{\frac{1}{2}} \quad \textrm{and}\quad {_{\A}\|x\|}:=\|{_{\A}\langle} x, x  \rangle\|^{\frac{1}{2}}
$$
coincide on $X_0$ and thus the quotient space $X_0/\|\cdot\|$ obtained by modding out the vectors of length zero with respect to the  seminorm $\|x\|:=\|x\|_{\A}={_{\A}\|x\|}$
is both a left and a right pre-Hilbert module over $\A$;
\item[iii)] denoting by $a\otimes b$ the canonical image of  $a\odot b$ in the quotient space $X_0/\|\cdot\|$ we have 
$$
ac\otimes b =a\otimes \HH(c)b,\quad \textrm{if }c\in \VV(\A),\qquad a\otimes cb =a\VV(c)\otimes b,\quad \textrm{if }c\in \HH(\A),
$$
and  $a\otimes b= a\VV(1)\otimes \HH(1)b$ for all $a,b \in \A$; 
\item[iv)] the inner-products in $X_0/\|\cdot\|$ satisfy  the imprimitivity condition.
\end{itemize}
\end{lem}  
\begin{proof}
i) All axioms of $\A$-valued semi-inner products for $\langle \cdot , \cdot  \rangle_\A$ and  ${_\A\langle} \cdot  , \cdot  \rangle$  except the non-negativity are straightforward, and to show the  latter  one may rewrite the proof of \cite[Proposition 5.2]{exel-inter}
(just erase the symbol $e_\HH$ or put $e_\HH=\HH(1)$).\\ 
ii) Similarly, the proof of \cite[Proposition 5.4]{exel-inter} implies that for $x=\sum_{i=1}^{n}a_i^*\odot b_i$, $a_i,b_i\in A$, we have 
\begin{equation}\label{banach space norm}
\|x\|_{\A}=\|\HH(aa^*)^{\frac{1}{2}} \HH(\VV(bb^*))^{\frac{1}{2}}\|=\|\VV(\HH(aa^*))^{\frac{1}{2}} \VV(bb^*)^{\frac{1}{2}}\|={_{\A}\|x\|}
\end{equation}
where $a=(a_1,...,a_n)^T$ and $b=(b_1,...,b_n)^T$ are viewed as column matrices.
\\ 
iii) For the  first part consult the proof of \cite[Proposition 5.6]{exel-inter}. The second part can be proved analogously. Namely, for every $x$, $y$, $a$, $b\in \A$ we have
$$
\langle x\otimes y, a\otimes b  \rangle_\A= y^*\HH(x^*a)b= y^*\HH(x^*a\VV(1))\HH(1)b=\langle x\otimes y, a\VV(1)\otimes \HH(1)b  \rangle_\A,
$$
which implies that $\|a\otimes b- a\VV(1)\otimes \HH(1)b\|=0$.\\
iv) The form of imprimitivity condition  allows us to check it only on simple tensors. Using iii), for $a,b,c,d,e,f\in A$, we have
\begin{align*}
 a\otimes  b \langle c\otimes d, e\otimes f  \rangle_\A &=a\otimes b d^* \HH(c^*e)f
 =a\otimes \HH(1)b d^* \HH(c^*e)f 
 \\
 &= a \VV\Big(\HH(1)b d^* \HH(c^*e)\Big)\otimes f = a \VV(\HH(1)b d^*) \VV(\HH(c^*e))\otimes f 
\\
 &= a \VV(b d^*) \VV(1)c^*e\VV(1)\otimes f = a \VV(b d^*) c^*e\otimes f
 \\
 &={_{\A}\langle} a\otimes  b,  c\otimes d \rangle  e\otimes f.   
\end{align*}
\end{proof}
\begin{defn}
We call the completion $X$ of the pre-Hilbert bimodule $X_0$ described in Lemma \ref{pre Hilbert bimodule construction}  a \emph{Hilbert bimodule associated to }$(\VV,\HH)$.
\end{defn}
\begin{rem}\label{exel's imprimitivity bimodule}
The Hilbert bimodule $X$ could be obtained directly from the  imprimitivity $\K_\VV$-$\K_\HH$-bimodule $\XX$ constructed   in  \cite[Section 5]{exel-inter}. Indeed, by \eqref{banach space norm}, $X$ and $\XX$ coincide as  Banach spaces, and since 
$$
\langle X, X \rangle_\A=\A\HH(1)\A, \qquad {_{\A}\langle} X, X \rangle=\A\VV(1)\A,
$$
$X$ can be considered as an imprimitivity $\A\VV(1)\A$-$\A\HH(1)\A$-bimodule. Furthermore,  the mappings $\lambda_\VV:\A\to \K_\VV$, $\lambda_\HH:\A\to \K_{\VV}$, the author of \cite{exel-inter} uses to define an $\A$-$\A$-bimodule structure on $\XX$, when restricted respectively to  $\A\VV(1)\A$ and $\A\HH(1)\A$ are isomorphisms. Hence we may use them to assume the identifications  $\K_\VV=\A\VV(1)\A$ and  $\K_\HH=\A\HH(1)\A$, and then Exel generalized correspondence  and the Hilbert bimodule $X$ coincide.
\end{rem}
Now we are ready to identify  the structure of  $C^*(A,\VV,\HH)$ as the Hilbert bimodule crossed product. 
\begin{prop}\label{prop on} We have a one-to-one correspondence between covariant representations $(\pi,S)$ of the interaction $(\VV,\HH)$ and covariant representations $(\pi,\pi_X)$ of the Hilbert bimodule $X$ associated to $(\VV,\HH)$. It is given by relations 
$$
\pi_X(a\otimes b)= \pi(a)S\pi(b),\,\, x \in X,\,\, \qquad S=\pi_X(1\otimes 1).
$$
In particular,  $C^*(A,\VV,\HH)\cong \A\rtimes_X \Z$ and the isomorphism is  gauge-invariant.
\end{prop}
\begin{proof} Let  $(\pi,S)$ be a covariant representation of $(\VV,\HH)$. Since  
\begin{align*}
\|\sum_{i}\pi(a_i)S\pi(b_i)\|^2&=\|\sum_{i,j}\pi(a_i)S\pi(b_ib^*_j)S^*\pi(a_j^*)\|=\|\pi\Big(\sum_{i,j}a_i\VV(b_ib^*_j)a^*_j\Big)\|
\\
&\leq\|\sum_{i}a_i\otimes b_i\|^2,
\end{align*}
we  see that $\pi_X(\sum_{i}a_i\otimes b_i):= \sum_{i}\pi(a_i)S\pi(b_i)$ defines a contractive linear mapping on $X_0/\|\cdot\|$. Clearly, it  satisfies 
 \eqref{covariant part one} and \eqref{covariant part two}. Hence by continuity  it extends uniquely  to $X$ in a way that $(\pi,\pi_X)$ is a covariant representation of $X$. Conversely suppose  that $(\pi,\pi_X)$ is a covariant representation  of the Hilbert bimodule $X$ and put $S:=\pi_X(1\otimes 1)$. Then for $a\in A$ we have
\begin{align*}
S\pi(a)S^*&=\pi_X((1\otimes 1) a)\pi_X(1\otimes 1)^*=\pi({_\A\langle} 1\otimes a,1\otimes 1\rangle)=\pi(\VV(a)).
\end{align*}
Similarly,
$
S^*\pi(a)S=\pi_X(1\otimes 1)^* \pi_X(a(1\otimes 1))=\pi(\langle 1\otimes 1,a\otimes 1\rangle_\A)=\pi(\HH(a))$.
\end{proof}
\begin{rem}
The Hilbert bimodule $X$ is nothing but the  GNS $C^*$-correspon\-dence determined by the completely positive map $\HH$, cf. \cite{kwa-exel}.  In particular, the above proposition shows that $C^*(A,\VV,\HH)$ is isomorphic to the  crossed product of $A$ by the completely positive map $\HH$ (or $\VV$, depending on preferences), see  \cite{kwa-exel}.
\end{rem}
Finally, by Remark  \ref{exel's imprimitivity bimodule} and Propositions \ref{normal crossed product as a generalized one}, \ref{prop on} we get 
\begin{cor}\label{prop on2}
Let  $\XX$ be the generalized correspondence constructed out of $(\VV,\HH)$ as in  \cite[Section 5]{exel-inter}.
The crossed product $C^*(A,\VV,\HH)$ for the interaction $(\VV,\HH)$ and the covariance algebra $C^*(\A,\XX)$ for $\XX$  are naturally isomorphic. 
\end{cor}
\subsection{Topological freeness, ideal structure and simplicity criteria} Let $(\VV,\HH)$ be a corner interaction. Since  $\VV(\A)$ and $\HH(\A)$ are hereditary  subalgebras of $A$ we have a standard way,   cf. e.g. \cite[Proposition 4.1.9]{Pedersen}, of  identifying their spectra with  open subsets of $\SA$. Namely, we assume that
\begin{equation}\label{identification equalities}
\widehat{\VV(\A)}=\{\pi \in \widehat{\A}: \pi(\VV(1))\neq 0\},\qquad \widehat{\HH(\A)}=\{\pi\in \widehat{\A}:\pi(\HH(1))\neq 0\}.
\end{equation}
The isomorphisms $\VV:\HH(\A)\to \VV(\A)$ and $\HH:\VV(\A)\to \HH(\A)$ induce mutually inverse homeomorphisms $\widehat{\VV}:\widehat{\VV(\A)}\to \widehat{\HH(\A)}$ and  $\widehat{\HH}:\widehat{\HH(\A)}\to \widehat{\VV(\A)}$, which under  identifications \eqref{identification equalities}  become partial homeomorphisms of $\SA$. 
\begin{defn}
We   refer  to   $\widehat{\VV}$ and $\widehat{\HH}$ as  \emph{partial homeomorphisms dual to $(\VV,\HH)$}.
\end{defn}
\begin{rem}\label{dual remarks}
For an irreducible representation $\pi:\A\to \B(H)$ with $\pi(\HH(1))\neq 0$ the element  $\widehat{\HH}(\pi)\in \SA$ is given by the (unique up to unitary equivalence) extension of the representation
$$
\widehat{\HH}(\pi)|_{\VV(\A)}=\pi\circ \HH: \VV(\A)\to \B(\pi(\HH(1))H), 
$$
Moreover, in the case $(\VV,\HH)$ is a  $C^*$-dynamical system $\HH(A)$ is an ideal and then $\pi(\HH(1))H=H$.
\end{rem}
\begin{prop}\label{coincidence of actions}
If $X$  is the Hilbert bimodule associated to  $(\VV,\HH)$ and  $X\dashind$ is the  partial homeomorphism of $\SA$ associated to $X$, then
$X\dashind=\widehat{\HH}. $

\end{prop}
\begin{proof}
Let $\pi:\A\to \B(H)$ be an irreducible representation with $\pi(\HH(1))\neq 0$. 
 For $(a\otimes b)\otimes_\pi h\in X\otimes_\pi H$, $a, b\in \A,\, h\in H$, using Lemma \ref{pre Hilbert bimodule construction} iii) we have  
\begin{align*}
X\dashind(\pi)(\VV(1))(a\otimes b)\otimes_\pi h&=\big(\VV(1)a\otimes b\big) \otimes_\pi h
=\big(\VV(1)a\VV(1)\otimes b\big) \otimes_\pi h
\\
&=(1\otimes \HH(a)b) \otimes_\pi h=
(1\otimes 1 )\otimes_\pi \pi(\HH(a)b) h. 
\end{align*}
Hence we see that  the space $H_0:=X\dashind(\pi)\big(\VV(1)\big) (X\otimes_\pi H)$  consists of  the vectors of the form $(1\otimes 1)\otimes_\pi h$, $h\in \pi(\HH(1))H$. Moreover,  for $h\in \pi(\HH(1))H$ we have
\begin{align*}
\|(1\otimes 1) \otimes_\pi h\|^2&=\langle (1\otimes 1) \otimes_\pi h, (1\otimes 1) \otimes_\pi h \rangle=\langle h,\pi(\langle 1\otimes 1,1\otimes 1\rangle_\A) h \rangle
\\
&= \langle h,\pi(\HH(1)) h \rangle
= \|h\|^2,
\end{align*}
and thus the mapping $(1\otimes 1)\otimes_{\pi} h\mapsto h$ is a unitary  $U$  from $H_0$ onto  $\pi(\HH(1))H$. For $a\in \VV(\A)$ we have
$$
 X\dashind(\pi)(a) (1\otimes 1)\otimes_{\pi} h= (a\otimes 1) \otimes_{\pi} h= (1\otimes \HH(a)) \otimes_{\pi} h=(1\otimes 1) \otimes_\pi \pi(\HH(a))h, 
$$ 
that is $X\dashind(\pi)(a)U^*h=U^*\pi(\HH(a))h$. 
 It follows that $U$ establishes unitary equivalence between $X\dashind(\pi):\VV(\A)\to \B(H_0)$
 and  $\pi\circ \HH:\VV(\A)\to \B(\pi(\HH(1))H)$. Hence $X\dashind=\widehat{\HH}$, cf. Remark \ref{dual remarks}.
\end{proof}
As $\widehat{\HH}=\widehat{\VV}^{-1}$, our preference  for $\widehat{\VV}$ in the sequel  is totally a subjective choice. 
\begin{thm}\label{interactions}
 Let $(\VV,\HH)$ a corner interaction and $\widehat{\VV}$ the partial homeomorphism dual to $\VV$.
\begin{itemize}
\item[i)]
If $\widehat{\VV}$ is  topologically
free, then   every representation of $C^*(A,\VV,\HH)$ which is faithful on $A$ is automatically faithful on  $C^*(A,\VV,\HH)$.
\item[ii)] If  $\widehat{\VV}$  is free, then $J \mapsto  \widehat{J \cap \A}$ is a lattice isomorphism
 between ideals in $C^*(A,\VV,\HH)$  and  open $\widehat{\VV}$-invariant sets in $\SA$.
 \item[iii)] If  $\widehat{\VV}$  is  topologically
free and minimal, then $C^*(A,\VV,\HH)$ is simple.
\end{itemize}
\end{thm}
\begin{proof}
Combine  Propositions \ref{prop on}, \ref{coincidence of actions} and Theorem \ref{main result}.
\end{proof}
\begin{rem}
Our simplicity criterion (Theorem \ref{interactions} iii)) have an intersection with the criteria in \cite[Theorems 4.1, 4.6]{Szwajcar} only in the case of  a $C^*$-dynamical system $(\alpha, \LL)$ where $\alpha$ is an isomorphism from  $A$ onto a full corner $\alpha(1)A\alpha(1)$ in $A$, cf. Remark \ref{remark on szwajcar} and Corollary \ref{cor for Jurgen} below. In this case topological freeness implies that no power of $\alpha$ or $\LL$ is inner (i.e. implemented by an isometry in $A$).
\end{rem}

In general, one can deduce  from Propositions \ref{prop on}, \ref{coincidence of actions}, see   \cite[discussion before Theorem 2.5]{kwa}, that  open $\widehat{\VV}$-invariant sets in $\SA$ are in a one-to-one correspondence with gauge invariant ideals in $C^*(A,\VV,\HH)$. Therefore it is useful to have a convenient  description of the former.
\begin{lem}\label{invariant ideals for endomorophisms}
Let $I$ be an ideal in $A$. The following conditions are equivalent: 
\begin{itemize}
\item[i)] The set $\widehat{I}\subset \SA$ is $\widehat{\VV}$-invariant,
\item[ii)] $\VV(I)=\VV(1) I\VV(1)$,
\item[iii)] $\VV(I)\subset I$ and  $\HH(I)\subset I$.
\end{itemize} 
\end{lem}
\begin{proof}
Notice that $\VV(1)I\VV(1)=I\cap\VV(A)$ and $\HH(1)I\HH(1)=I\cap\HH(A)$. Hence $\VV(I)=\VV(\HH(1)I\HH(1))=\VV(I\cap \HH(A))$ and ii) reads as $\VV(I\cap \HH(A))=I\cap \VV(A)$. Now equivalence i)$\Leftrightarrow$ ii) is clear.
\\
ii)$\Rightarrow$ iii).  We have  $\VV(I)=\VV(1) I\VV(1)\subset I$  and  $\HH(I)=\HH(\VV(1)I\VV(1))=\HH(\VV(I))=\HH(1)I\HH(1)\subset I$.
\\
iii)$\Rightarrow$ ii). The inclusion $\VV(I)\subset I$  implies  $\VV(I)\subset \VV(1)I\VV(1)$ and $\HH(I)\subset I$ implies that $\VV(1)I\VV(1)=\VV(\HH(I))\subset \VV(I)$.
\end{proof}
\begin{cor}\label{cor for Jurgen}
Suppose $(\alpha, \LL)$ is a corner $C^*$-dynamical system. The partial homeomorphism $\widehat{\alpha}$ is minimal if and only if 
there is no nontrivial ideal $I$ in $A$ such that $\alpha(I)\subset I$.
\end{cor}
\begin{proof} The  if part follows immediately from  Lemma \ref{invariant ideals for endomorophisms}. If we suppose  that  $\alpha(I)\subset I$ and $\alpha(I)\neq \alpha(1)I\alpha(1)$ for a certain nontrivial ideal $I$ in $A$, then one sees (by induction on $n$) that the closure $J$ of elements of the form $\sum_{k=0}^n \mathcal{L}^k(a_k)$,  $a_k \in I$, $k=0,...,n$, $n\in \N$, is a nontrivial ideal in $A$ (it does not contain the unit) such that $\alpha(J)\subset J$ and $\mathcal{L}(J)\subset J$. Hence by Lemma \ref{invariant ideals for endomorophisms}, $\widehat{\alpha}$ is not minimal.
\end{proof}
\subsection{$K$-theory} 
We retain the notation from page \pageref{Katsura Pimsner voiculescu sequence} with the  additional assumption that   $X$ is the Hilbert bimodule associated to a corner interaction $(\VV,\HH)$. In particular, ${_A\langle X,X\rangle}=A\VV(1)A$.

\begin{lem}\label{probably non-standard proposition}
The  following diagram commutes and the horizontal map is an isomorphism
$$
\begin{xy}
\xymatrix{
      K_*(\VV(A)) \ar[rr]^{\iota_*} \ar[dr]_{(\iota_{22}\circ \HH)_*} &   &  \ar[dl]^{(\iota_{11}\circ \phi)_*} K_*(A\VV(1)A)
      \\
       & K_*(D_{X}) & }
  \end{xy} 
 $$ 
\end{lem}
\begin{proof}
Since $\VV(A)$ is a full corner in $A\VV(1)A$ it is known that the inclusion $\iota:\VV(A)\to A\VV(1)A$ yields isomorphisms of $K$-groups, cf. e.g. \cite[Proposition B.5]{katsura}. %
We claim  that the map
$$
M_2(\HH(A))\ni \left(\begin{matrix} 
a_{11} & a_{12}
\\
a_{21} & a_{22}
\end{matrix}\right) \stackrel{\Phi}{\longmapsto} \left(\begin{matrix} 
\phi(\VV(a_{11})) & 1\otimes a_{12}
\\
\flat(1\otimes a_{21}^*) & a_{22}
\end{matrix}\right) \in D_X,
$$
where $\flat:X \to  \X$ is the canonical antilinear isomorphism, is a homomorphism of $C^*$-algebras. Plainly, it is  linear, $*$-preserving, and  the reader easily checks that  $\Phi(ab)=\Phi(a)\Phi(b)$, for $a=[a_{ij}], b=[b_{ij}]\in M_2(\HH(A))$,   using the following calculations
\begin{align*}
(1\otimes a_{12})\cdot \flat(1\otimes b_{21}^*)x\otimes y&=\Theta_{(1\otimes a_{12}),(1\otimes b_{21}^*)}x\otimes y=1\otimes a_{12}b_{21}\HH(x)y
\\
&=\VV(a_{12}b_{21})\VV(\HH(x))\otimes y=\VV(a_{12})\VV(b_{21}) x\otimes y
\\
&=\phi(\VV(a_{12} b_{21})) (x\otimes y),
\end{align*}
$$
\phi(\VV(a_{11}))(1\otimes b_{12})=\VV(a_{11})\otimes  b_{12}=1\otimes a_{11}b_{12},
$$
$$
\flat(1\otimes a_{21}^*)\cdot (1\otimes b_{12})=\langle 1\otimes a_{21}^*,1\otimes b_{12}\rangle_{A}= a_{21}\HH(1)b_{12}=a_{21}b_{12}.
$$
This shows our claim. The following diagram commutes (it commutes on the level of $C^*$-algebras)
$$
\begin{xy}
\xymatrix{
      K_*(\VV(A)) \ar[rr]^{\iota_*}  \ar[d]_{(\iota_{11}\circ \HH)_*} &   &  \ar[d]^{(\iota_{11}\circ \phi)_*} K_*(A\VV(1)A) 
      \\
     K_*( M_2(\HH(A))) \ar[rr]^{\Phi_*}&  & K_*(D_{X})}
  \end{xy}. 
 $$ 
However, since for any $C^*$-algebra $B$ the homomorphisms $\iota_{ii}:B\to M_2(B)$,  $i=1,2$, induce the same mappings on the level of $K$-theory,  the mappings $(\iota_{11}\circ \HH)_*, (\iota_{22}\circ \HH)_*:K_*(\VV(A))\to K_*(M_2(\HH(A)))$ coincide. Moreover, by the form of $\Phi$ we see that  $\Phi\circ\iota_{22}\circ \HH=\iota_{22}\circ \HH$ on $\VV(A)$. Hence
 $$
 (\iota_{11}\circ \phi)_* \circ\iota_* = \Phi_*\circ (\iota_{11}\circ \HH)_*=\Phi_*\circ (\iota_{22}\circ \HH)_*=(\iota_{22}\circ \HH)_*.
 $$
\end{proof}
Using the above lemma we see that in  sequence \eqref{Katsura Pimsner voiculescu sequence} we may replace  $\K_*({_A\langle X,X\rangle})=K_*(A\VV(1)A)$ with $K_*(\VV(A))$ and then $X_*$ turns into $(\iota_{22})_*^{-1}\circ (\iota_{22}\circ \HH)_*= \HH_*$. Hence we get the following version of Pimsner-Voiculescu exact sequence, cf. \cite{Paschke1}, \cite{Rordam}.
\begin{thm}\label{Voicu-Pimsner for interacts} For any corner interaction $(\VV,\HH)$ we have the following exact sequence  
$$
\begin{xy}
\xymatrix{
      K_0(\VV(A)) \ar[r]_{\iota_*- \HH_*} & K_0(A) \ar[r]_{(i_A)_*\,\,\,\,\,\,\,\,\,}  &   \ar[d]  K_0(C^*(A,\VV,\HH))
             \\
   K_1(C^*(A,\VV,\HH)) \ar[u]  &  K_1(A)  \ar[l]_{\,\,\,\,\,\,\,\,\,\,\,\,(i_A)_*}  &   \ar[l]_{\iota_*- \HH_*}  K_1(\VV(A))
              } 
  \end{xy} .
$$
\end{thm}
\section{Graph $C^*$-algebras via interactions}\label{Cuntz-Krieger algebras section}

 In this section we introduce and study  properties of  graph interactions. We show that Theorem \ref{interactions} applied to graph interactions is equivalent to the  Cuntz-Krieger uniqueness theorem and its consequences. We  use  Theorem \ref{Voicu-Pimsner for interacts} to calculate $K$-theory for graph algebras straight from the dynamics on their AF-cores.

\subsection{Graph $C^*$-algebra  $C^*(E)$ and its AF-core}\label{graph core subsection}
 Throughout we let $E = (E^0,E^1, r, s)$ to be  a fixed finite  directed graph. Thus  $E^0$ is a set of vertices, $E^1$ is a set of edges,    $r,s:E^{1}\to E^0$ are range, source maps,  and we assume that both sets $E_0$, $E^1$ are finite. We write $E^{n}$, $n>0$, for the set of paths $\mu= \mu_1\dots\mu_n$,  $\mu_i\in E^{1}$, $r(\mu_i)=s(\mu_{i+1})$, $i=1,...,n-1$, of length $n$. The maps $r$, $s$ naturally extend to $E^{n}$, so that $(E^0,E^n,s,r)$ is the graph, and $s$ extends to the  set
$
E^{\infty}$ 
of infinite paths 
$\mu = \mu_1\mu_2\mu_3...\,$.  We also put $s(v)=r(v)=v$ for $v\in E^0$. The elements of $E^0_{sinks}:=E^0\setminus s(E^1)$  and respectively  $E^0_{sources}:=E^0\setminus r(E^1)$ are called sinks and sources. We also consider sets $
E^n_{sinks}=\{\mu\in E^n: r(\mu) \in E^0_{sinks}\}$, $n\in \N$.

We adhere to conventions of  \cite{kum-pask-rae}, \cite{bprs}.
In our setting a \emph{Cuntz-Krieger $E$-family} compose of non-zero pair-wise orthogonal projections  $\{P_v: v\in E^0\}$   and   partial isometries $\{S_{e}:e\in E^1\}$  satisfying 
$$ 
S_e^* S_e=P_{r(e)}\quad\textrm{and}\quad  P_v = \sum_{e \in s^{-1}(v)} S_eS_e^*\qquad \textrm{for all}\quad v\in s(E^1),\, e\in E^1.
$$
Having such a family we put $S_\mu := S_{\mu_1}S_{\mu_2}\dotsm
S_{\mu_n}$ for  $\mu=\mu_1...\mu_n$ ($S_\mu\neq 0 \Rightarrow \mu\in E^n$)    and  $S_v:=P_v$ for $v\in E^0$. The above Cuntz-Krieger relations extend to operators $S_\mu$, see \cite[Lemma 1.1]{kum-pask-rae}, as follows
$$
S^*_\nu S_\mu=
\begin{cases} 
S_{\mu'}, & \textrm{if } \mu=\nu\mu',\,\, \mu'\notin E^0, \\
  S^*_{\nu'} & \textrm{if } \nu=\mu\nu', \,\,\nu'\notin E^0, \\
  0 & \textrm{otherwise. } 
\end{cases} 
$$
In particular, $C^*(\{P_v: v\in E^0\} \cup \{S_{e}:e\in E^1\})$ is the closure of the linear span of elements  $S_\mu S^*_\nu$, $\mu\in E^n$, $\nu\in E^m$, $n,m \in \N$.

 The \emph{graph $C^*$-algebra}   $C^*(E)$  of $E$ is   a universal $C^*$-algebra generated by  a universal Cuntz-Krieger $E$-family $\{s_e:e\in E^1\}$, $\{p_v: v\in E^0\}$. It  is equipped with the natural \emph{circle gauge action} 
$\gamma:\T \to \Aut C^*(E)$ established by relations $\gamma_\lambda(p_v)=p_v$,  $\gamma_\lambda(s_e)=\lambda s_e$,  for $v\in E^0$, $e\in E^1$, $\lambda\in \T$.
 The   fixed point $C^*$-algebra for $\gamma$ is called the \emph{core}. It is an AF-algebra  of the form
$$
\FF_E:=\clsp\left\{ s_\mu s^*_\nu: \mu,\nu\in E^n, \ n=0,\,1,\,\dotsc\right\}.
$$
 We  recall the   standard  Bratteli diagram for $\FF_E$. For each vertex $v$ and $N\in \N$ we set
$$
\FF_N(v):=\spane\{ s_\mu s^*_\nu: \mu, \nu \in E^N,\,\, r(\mu)=r(\nu)=v \},
$$
which is a simple $I_n$  factor with $n=|\{\mu \in E^N: r(\mu)=v\}|$   (if $n=0$  we put $\FF_N(v):=\{0\}$).  The spaces 
$$
\FF_N:=\Big(\oplus_{\, v\notin E^0_{sinks}} \FF_N(v)\Big)\oplus \Big(\oplus_{\,w\in E^{0}_{sinks}} \oplus_{i=0}^N \FF_i(w)\Big),\qquad N\in \N,
$$
 form an increasing family of finite-dimensional algebras, cf. e.g. \cite{bprs}, and 
\begin{equation}\label{inductive limit}
\FF_E=\overline{\bigcup_{N\in \N} \FF_N}
.
\end{equation}
We  denote by $\Lambda(E)$ the corresponding Bratteli diagram for $\FF_E$. If $E$ has no sinks we can view  $\Lambda(E)$ as an infinite vertical concatenation of  $E$: 
  on the $n$-th level we have the vertices $r(E^n)$, $n\in \N$,  and multiplicities are given by the number of edges with corresponding  endings and sources. If  $E$ has  sinks,  one has to attach to every sink on each level an infinite tail, so  on the $n$-th level of $\Lambda(E)$  we have $r(E^n)\cup \bigcup_{k=0}^{N-1}\{v^{(k)}: v \in r(E^{k}_{sinks})\}$ and each $v^{(k)}$ descends to $v^{(k)}$ with multiplicity one.   
  We adopt the convention that if $V$ is a subset of $E^0$ we treat it as a full subgraph of $E$ and $\Lambda(V)$ stands for the corresponding Bratteli diagram  for $\FF_V$. In particular, if $V$ is  \emph{hereditary}, i.e. $s(e)\in V\Longrightarrow r(e)\in V$ for all $e\in E^1$, and  \emph{saturated}, i.e.  every vertex  which feeds into $V$ and only $V$ is in $V$, then  the subdiagram $\Lambda(V)$ of $\Lambda(E)$ yields an ideal in $\FF_E$ which is naturally identified with $\FF_V$.  In general, viewing $\Lambda(E)$ as an infinite directed graph the hereditary and saturated subgraphs (subdiagrams) of $\Lambda(E)$ correspond to ideals in $\FF_E$, see \cite[3.3]{Bratteli}.

 \subsection{Interactions arising from graphs}
For each vertex $v\in E^0$  we let 
 $
 n_v:=|r^{-1}(v)|
 $
 be the number of the edges received by $v$. We define an operator $s$ in $C^*(E)$ as the sum of the partial isometries $\{s_{e}:e\in E^1\}$  "averaged" on the spaces  corresponding to projections $\{p_{v}:v\in r(E^0)\}$ that are not sources:
\begin{equation}\label{partial isometry for graph}
s:= \sum\limits_{e\in E^1 }\frac{1}{ \sqrt{n_{r(e)}}}\,\, s_e=\sum\limits_{v\in r(E^1)}\frac{1}{ \sqrt{n_{v}}} \sum\limits_{e\in r^{-1}(v)}\,\, s_e.
\end{equation}
Since $
s^*s=\sum_{v \in r(E^1)}p_v$ is a projection the operator $s$ is   a partial isometry. It  is an isometry iff $E$ has no sources. We use $s$ to define  
\begin{equation}\label{dS}
\VV(a) := s a s^*,\qquad \HH(a):=s^*a s, \qquad a\in C^*(E).
\end{equation}
Plainly, $(\VV,\HH)$ is a corner interaction over $C^*(E)$. Moreover, one  sees that     $\VV$ and  $\HH$ are unique bounded linear maps on $C^*(E)$ satisfying the following formulas
\begin{equation}\label{endomorphism action on F_A}
 \VV \Big(s_{\mu} s^*_{\nu}\Big)=\begin{cases} \frac{1}{ \sqrt{n_{s(\mu)}n_{s(\nu)}}} \sum\limits_{e,f\in E^1} s_{e\mu} s^*_{f\nu}, &  n_{s(\mu)}n_{s(\nu)}\neq 0, \\
 0 , &  n_{s(\mu)}n_{s(\nu)}= 0,  
 \end{cases}
\end{equation}
\begin{equation}\label{transfer action on F_A}
\HH \Big(s_{e\mu} s^*_{f\nu}\Big)= \frac{1}{ \sqrt{n_{s(\mu)}n_{s(\nu)}}} \,\, s_{\mu} s^*_{\nu},\,\,\textrm{  }  
\HH \Big(p_v\Big)=\begin{cases} 
\sum\limits_{e\in s^{-1}(v)}\frac{p_{r(e)}}{n_{r(e)}}, &  v\notin E^0_{sinks}, \\
 0 , & v\in E^0_{sinks},
 \end{cases}
\end{equation}
 where $\mu\in E^{n}$, $\nu\in E^{m}$, $n,m\in \N$, $e,f\in E^1$, $v\in E^0$. It follows that $\VV$ and $\HH$ preserve the core  algebra $\FF_E$. Hence $(\VV,\HH)$ defines a corner interaction over $\FF_E$.  We note, however,  that $\VV$ hardly ever preserves the canonical diagonal algebra $\D_E:=\clsp\left\{ s_\mu s^*_\mu: \mu\in E^n, n\in \N\right\}\subset \FF_E$. 
\begin{defn}
We call the pair $(\VV,\HH)$ of continuous linear maps on $\FF_E$ satisfying \eqref{endomorphism action on F_A}, \eqref{transfer action on F_A}  a  \emph{(corner) interaction of the graph $E$} or simply a \emph{graph interaction}.
\end{defn}

The following  statement is one of the facts justifying the above definition. 
\begin{prop}\label{thm for Huef and Raeburn}
We have  a one-to-one correspondence between Cuntz-Krieger $E$-families  $\{P_v: v\in E^0\}$, $\{S_{e}:e\in E^1\}$ for $E$ and faithful covariant representations $(\pi,S)$ of the graph interaction $(\VV,\HH)$. It is given by the relations 
$$
S= \sum\limits_{e\in E_1 }\frac{1}{ \sqrt{n_{r(e)}}} S_e, \qquad P_v=\pi(p_{v}),\qquad   S_e= \sqrt{n_{r(e)}} \pi(s_e s_e^*)S. 
$$ 
 In particular, we have a  gauge-invariant isomorphism
$
C^*(E)\cong C^*(\FF_E,\VV,\HH)$.
\end{prop}
\begin{proof} A Cuntz-Krieger $E$-family  $\{P_v: v\in E^0\}$, $\{S_{e}:e\in E^1\}$ yields a representation $\pi$ of $C^*(E)$ which is well known to be faithful on $\FF_E$. By the  definition of $(\VV,\HH)$ the pair $(\pi|_{\FF_E}, S)$ where $S:=\pi(s)=\sum\limits_{e\in E_1 }\frac{1}{ \sqrt{n_{r(e)}}} S_e$ is a covariant representation of $(\VV,\HH)$. 
 Conversely, let $(\pi,S)$ be a faithful representation of  $(\VV,\HH)$ and put $P_v:=\pi(p_{v})$ and $S_e:= \sqrt{n_{r(e)}} \pi(s_e s_e^*)S$. We claim that  $\{P_v: v\in E^0\}$, $\{S_{e}:e\in E^1\}$ is  a Cuntz-Krieger $E$-family such that $S= \sum\limits_{e\in E_1 }\frac{ S_e}{ \sqrt{n_{r(e)}}}$. Indeed, for $e\in E^1$  we have
\begin{align*}
S_e^* S_e 
=n_{r(e)}\pi(p_{r(e)})\pi(\HH(s_es_e^*))\pi(p_{r(e)})=\pi(p_{r(e)}) =P_{r(v)},
\end{align*}
and for $v\in s(E^1)$ we have
\begin{align*}
\sum_{e\in s^{-1}(v)} S_e S_e^*&=\sum_{e\in s^{-1}(v)}n_{r(e)} \pi(s_e s_e^*)\pi(\VV(1))\pi(s_e s_e^*)
\\
&=\sum_{e\in s^{-1}(v), e_1,e_2 \in E^1}  \frac{n_{r(e)}}{\sqrt{n_{r(e_1)}n_{r(e_2)}} }\pi(s_es_e^*(s_{e_1} s_{e_2}^*)s_es_e^*)
\\
&=\sum_{e\in s^{-1}(v)}  \pi(s_es_e^*)=\pi(p_{v})=P_{v}.
\end{align*}
Now note  that $
S^*S=\pi(\HH(1))=\sum_{v \in r(E^1)}\pi(p_{v})$ and thus $S=\sum_{e\in E^1}S \pi(p_{v})$. 
Moreover, for each    $v\in r(E^1)$ we have 
\begin{align*}
\left(\sum\limits_{e\in r^{-1}(v)}\pi(s_e s_e^*)  \right) S\pi(p_v)S^* &= \sum\limits_{e\in r^{-1}(v)}\pi(s_es_e^*\VV(p_v))
  \\
 & = \sum\limits_{e, e_1,e_2 \in r^{-1}(v)}\frac{ \pi(s_es_e^* s_{e_1}s_{e_2}^*)}{ n_{v}}
 \\
 &= \sum\limits_{e_1,e_2 \in r^{-1}(v)}\frac{\pi(s_{e_1}s_{e_2}^*)}{ n_{v}} =\pi(\VV(p_v))
 \\
&=S\pi(p_v)S^*.
\end{align*}
Hence the final space of the partial isometry $S\pi(p_{v})$  decomposes into the orthogonal sum of ranges of  the projections  $\pi(s_es_e^*)$, $e\in r^{-1}(v)$, and consequently   
$$
\sum\limits_{e\in E^1 }\frac{ S_e}{ \sqrt{n_{r(e)}}}= \sum\limits_{e\in E^1 } \pi(s_es_e^*)S\pi(p_{r(e)}) =\sum\limits_{v\in E^0 } \sum\limits_{e\in r^{-1}(v) } \pi(s_es_e^*)S\pi(p_{v})= S.
$$
\end{proof}
\begin{rem}\label{remark for astrid}
 If $E$ has no sources then $s$   is an isometry and  $\VV$ is an injective  endomorphism with hereditary range. In this case  $C^*(E)$ coincides with various crossed products by endomorphisms that involve isometries, cf. \cite{Ant-Bakht-Leb}, \cite{exel2}, \cite{Paschke0}. In particular, Proposition \ref{thm for Huef and Raeburn}  has a nontrivial  intersection   with   \cite[Theorem 5.2]{hr} proved for locally finite graphs without sources.
\end{rem} 
\begin{rem}
The canonical completely positive map $\phi_E:C^*(E)\to C^*(E)$ is given by the formula
$$
\phi_E(x)=\sum_{e\in E^1} s_e x s_e^*.
$$
This map (unlike $\VV$ but like $\HH$) always preserves both $\FF_E$ and $\D_E$ and the pair $(\phi_E,\HH)$ is a $C^*$-dynamical on $\D_E$, cf. Proposition \ref{powers of partial isometry} below. Moreover,  if $E$ has no sinks  the same relations as in Proposition \ref{thm for Huef and Raeburn} yield an isomorphism between $C^*(E)$ and the Exel's crossed product $\D_E\rtimes_{(\phi_E,\HH)} \N$, see  \cite[Theorem 5.1]{BRV}. The advantage of   $\D_E\rtimes_{(\phi_E,\HH)} \N$ over  $C^*(\FF_E,\VV,\HH)$ is that it starts from a commutative $C^*$-algebra $\D_E$. The disadvantages are that the dynamics in $(\phi_E,\HH)$ is irreversible and  involves two mappings (at least implicitly, see \cite{kwa-exel}), while in essence  $(\VV,\HH)$ is a single map (recall Proposition \ref{proposition about uniqueness of interactions}) possessing a natural generalized inverse.
\end{rem}
A natural question to ask is when the graph interaction $(\VV,\HH)$ is a $C^*$-dynamical system. It is somewhat surprising that this holds only if $(\VV,\HH)$ is a part of a group interaction. We take up the rest of this subsection to clarify  this issue in detail. To this end we will use a partially-stochastic matrix $P=[p_{v,w}]$  arising from  the adjacency matrix $A_E=[A_E(v,w)]_{v,w \in E^0}$ of the graph $E$. Namely, we let 
\begin{equation}\label{stochastic matrix}
 p_{v,w}:=\begin{cases} 
 \frac{A_E(v,w)}{n_w}, & A_E(v,w)\neq 0, \\
0 , & A_E(v,w)= 0,
\end{cases}
\end{equation}
where $A_E(v,w)=|\{e\in E^1: s(e)=v, r(e)=w\}|$. By a   \emph{partially-stochastic matrix} we mean a non-negative matrix in which each non-zero column sums up to one.

\begin{prop}\label{powers of partial isometry}
Let $s$ be the operator  given by \eqref{partial isometry for graph}  and let $n \geq 1$. The following conditions are equivalent:
\begin{itemize}
\item[i)] $(\VV^n,\HH^n)$ is an interaction over $\FF_E$,
\item[ii)] $(\phi_E^n,\HH^n)$ is a $C^*$-dynamical system on  $\D_E$,
\item[iii)] operator  $s^n$ is a partial isometry, 
\item[iv)] $n$-th power of the matrix  $P=\{p_{v,w}\}_{v,w\in E^0}$  is  partially-stochastic,
\item[v)] for any $\mu\in E^n$ and $\nu\in E^k$, $k<n$, such that  $r(\mu)=r(\nu)$ we have $s(\nu)\notin E^0_{sources}$.  
 
\end{itemize}
\end{prop}
\begin{proof}   
i) $\Leftrightarrow$ iii). As $\VV^n(\cdot)=s^n(\cdot)s^{*n}$ and $\HH^n(\cdot)=s^{*n}(\cdot)s^{n}$ one readily checks  that iii) implies i), and if we assume i) then $s^n$ is a partial isometry because  $\HH^n(1)$ is a   projection by Lemma \ref{lem stanislaw}. 
\\
iii) $\Leftrightarrow$ iv). Operator  $s^n$ is a partial isometry iff $\HH^n(1)$ is a projection. Since $\HH(p_v)=\sum_{w \in E^0}p_{v,w} p_w$, cf. \eqref{transfer action on F_A},   we get 
$$
\HH^n(1)=\sum_{v_0,...,v_n\in E^0} p_{v_0,v_1}\cdot  p_{v_1,v_2}\cdot  ...\cdot p_{v_{n-1},v_{n}}  p_{v_n} =\sum_{v,w\in E^0} p_{v,w}^{(n)} p_w
$$
where $P^n=\{p_{v,w}^{(n)}\}_{v,w\in E^0}$ stands for the $n$-th power of $P$. By the orthogonality of projections $p_w$,  it follows that  $\HH^n(1)$ is a projection iff   $\sum_{v\in E^0} p_{v,w}^{(n)}\in \{0,1\}$ for all $w\in E^0$, that is iff  $P^{n}$ is partially-stochastic. 
\\
ii) $\Leftrightarrow$ iv). We know that $\phi_E:\D_E\to \D_E$ is an endomorphism and $\HH$ is its transfer operator. Moreover, it is a straight forward fact that an iteration of an endomorphism and its transfer operator gives  again an endomorphism and its transfer operator. Thus $(\phi_E^n, \HH^n)$ is a $C^*$-dynamical system iff  the transfer operator $\HH^n$ is regular, that is iff $\phi_E^n(\HH^n(1))=\phi_E^n(1)$. 
However, as
$$
\phi_E^n(\HH^n(1))=\sum_{v,w\in E^0} p_{v,w}^{(n)} \phi_E^n(p_w)=\sum_{v\in E^0, \mu \in E^n} p_{v,r(\mu)}^{(n)} s_\mu s_\mu^*$$
and $\phi_E^n(1)=\sum_{\mu \in E^n} s_\mu s_\mu^*
$
we see that $\phi_E^n(\HH^n(1))=\phi_E^n(1)$ if and only if $P^n=\{p_{v,w}^{(n)}\}_{v,w\in E^0}$ is partially-stochastic.
\\ 
iv) $\Rightarrow$ v). Assume that v) is not true, that is let $\mu\in E^n$ and $\nu\in E^k$, $k<n$, be such that  $r(\mu)=r(\nu)$ and $s(\nu)\in E^0_{sources}$. Notice that   the condition $\sum_{v\in E^0} p_{v,w}^{(n)} >0$ is equivalent to existence of   $\eta \in E^n$ such that  $w=r(\mu)$. Hence putting $w:=r(\mu)=r(\nu)$ and $v_0:=s(\nu)$ 
we have $\sum_{v\in E^0} p_{v,w}^{(n)} >0$ and $p_{v_{0},w}^{(k)}> 0$. Then $\sum_{v\in E^0}p_{v,v_{0}}^{(n-k)}=0$ (because $v_0\in E^0_{sources}$) and therefore 
$$
0< \sum_{v\in E^0} p_{v,w}^{(n)}=\sum_{v\in E^0,\atop v_{n-k}\in E^0\setminus\{v_0\}} p_{v,v_{n-k}}^{(n-k)}p_{v_{n-k},w}^{(k)}\leq \sum_{v_{n-k}\in E^0\setminus\{v_0\}} p_{v_{n-k},w}^{(k)}< 1,
 $$
that is $P^n$ is not partially-stochastic. 
\\
v) $\Rightarrow$ iv). Suppose that $\sum_{v\in E^0} p_{v,w}^{(n)} >0$. By our assumption for each $0<k<n$ the condition   
$p_{v_{k},w}^{(n-k)}\neq 0$  implies that $v_k\notin  E^0_{sources}$. However, relation $v_k\notin  E^0_{sources}$ is equivalent to  $\sum_{v_{k-1}\in E^0} p_{v_{k-1},v_{k}}^{(1)}=1$ (because $P$ is partially-stochastic). Therefore proceeding inductively for $k=1,2,3...,n-1$ we get
$$
\sum_{v\in E^0} p_{v,w}^{(n)} = \sum_{v_0,v_{1}\in E^0} p_{v_0,v_{1}}^{(1)} p_{v_{1},w}^{(n-1)}= \sum_{v_1\in E^0}  p_{v_{1},w}^{(n-1)}=...=\sum_{v_{n-1}\in E^0}  p_{v_{n-1},w}^{(1)}=1.
$$
\end{proof}
\begin{ex}
It follows from  Proposition \ref{powers of partial isometry} that if we consider a graph interaction $(\VV,\HH)$ arising from  the following graph 
\\
\setlength{\unitlength}{.8mm}
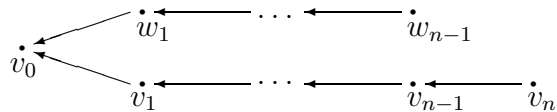
\begin{picture}(80,20)(-5,0)
\put(35,10){\circle*{1}} \put(33,6){$v_0$}
\put(55,16){\circle*{1}} \put(54,12){$w_1$}
\put(74,15){$\dots$} \put(74,4){$\dots$}
\put(100,16){\circle*{1}} \put(99,12){$w_{n-1}$}
\put(100,4){\circle*{1}} \put(99,0.5){$v_{n-1}$}
\put(120,4){\circle*{1}} \put(119,0.5){$v_{n}$}
\put(55,4){\circle*{1}} \put(53,0.5){$v_1$}
\put(53,16){\vector(-3,-1){16}} \put(98,16){\vector(-1,0){16}} \put(98,4){\vector(-1,0){16}} \put(118,4){\vector(-1,0){16}} 
\put(53,4){\vector(-3,1){16}} \put(73,16){\vector(-1,0){16}} \put(73,4){\vector(-1,0){16}}
\end{picture}\\
then $(\VV,\HH)$ has the property that its $k$-th  power $(\VV^k,\HH^k)$, for $k>1$,  is an interaction unless $k=n$. Hence   by considering a disjoint sum of graphs of the  above form one can obtain a graph interaction with an arbitrary finite distribution of powers being  interactions.
\end{ex}
In our specific situation of graph interactions we may  prolong the list of equivalent conditions in  Proposition \ref{characterization of corner interactions} as follows.
\begin{cor}
\label{condition on matrices and partial isometries}
Let $(\VV,\HH)$ be the interaction associated to the graph $E$. The following conditions are equivalent:
\begin{itemize}
\item[i)]  $(\VV,\HH)$ is  a $C^*$-dynamical system, 
\item[ii)] $(\VV^n,\HH^n)$ is an interaction for all $n\in \N$,
\item[iii)] $(\phi_E^n,\HH^n)$ is a $C^*$-dynamical system for all $n\in \N$,
\item[iv)]  operator $s$  given by \eqref{partial isometry for graph} is a power partial isometry,
\item[v)]  every power of  the  matrix $P=\{p_{v,w}\}_{v,w\in E^0}$ is partially-stochastic,
\item[vi)] every two  paths in $E$ that have the same length and  the same ending either both starts in sources  or not in sources.
\end{itemize}
\end{cor}
\begin{proof}
Item vi) holds if and only if item v) in Proposition \ref{powers of partial isometry} holds for all $n\in \N$. 
Hence  by Proposition \ref{powers of partial isometry} we get the equivalence between all the items from ii) to vi) in the present assertion. Furthermore, we recall that $\HH(1)=s^*s=\sum_{v \in r(E^1)} p_v$, and item i) is equivalent to $\HH(1)$ being a central element in $\FF_E$, see  Proposition \ref{characterization of corner interactions}. Hence the equivalence  i)$\Leftrightarrow$ vi)  follows  from the relations 
$$
\HH(1) s_\mu s^*_\nu=\begin{cases}
0, & \textrm{if } s(\mu) \notin r(E^1)\\
s_\mu s^*_\nu, & \textrm{otherwise }
\end{cases},\,\, 
s_\mu s^*_\nu \HH(1) =\begin{cases}
0, & \textrm{if } s(\nu) \notin r(E^1)\\
s_\mu s^*_\nu, & \textrm{otherwise }
\end{cases}, 
$$
which hold for all $\mu,\nu\in E^n$, $n\in \N$.
\end{proof}
A natural question to ask is when $\HH$  is multiplicative. We rush to say that it is  hardly the case.
\begin{prop}
The pair $(\HH,\VV)$, where $(\VV,\HH)$ is the interaction of  $E$, is a $C^*$-dynamical system if and only if  the mapping $r:E^1\to E^0$ is injective.
\end{prop}
\begin{proof} By  Proposition \ref{characterization of corner interactions} multiplicativity of $\HH$ is equivalent to $\VV(1)$ being a central element in $\FF_E$. If  $r:E^1\to E^0$ is injective, then $\FF_E=\D_E$ is commutative and  $(\HH,\VV)$ is a $C^*$-dynamical system because $\VV(1)\in \FF_E$. Conversely,  let us assume that the projection  $\VV(1)=ss^*= \sum\limits_{v\in r(E^1)}\frac{1}{n_{v}} \sum\limits_{e,f\in r^{-1}(v)}\,\, s_es_f^*$  is  central in $\FF_E$ and let   $g,h\in E^1$ be such that $r(g)=r(h)=v$. Since   
$$\VV(1) s_g s^*_h= \frac{1}{n_{v}}\sum_{e \in r^{-1}(v)}  s_{e} s^*_{h}, \quad\textrm{ and }\quad s_g s^*_h\VV(1)=\frac{1}{ n_{v}}\sum_{f \in r^{-1}(v)}  s_{g} s^*_{f}
$$ 
we have $\sum_{e \in r^{-1}(v)}  s_{e} s^*_{h}=\sum_{f \in r^{-1}(v)}  s_{g} s^*_{f}$, which implies $g=h$. Hence  $r:E^1\to E^0$ is injective. 
\end{proof}

   \subsection{Dynamical systems dual to  graph interactions.}
   Let $(\VV,\HH)$ be the  interaction of the graph  $E$.    We  obtain a satisfactory picture of the  system dual  to $(\VV,\HH)$ using a Markov shift $(\Omega_E,\sigma_E)$ dual to the commutative system $(\D_E, \phi_E)$. Namely, we put $\Omega_E=\bigcup_{N=0}^{\infty}E^N_{sinks}\cup E^\infty$ and   let 
   $
   \sigma_E:\Omega_E\setminus E^0_{sinks} \to \{\mu \in \Omega_E: s(\mu) \notin E^0_{sources}\}
   $ 
   be the shift defined by the formula
$$
\sigma_E(\mu)=\begin{cases}
\mu_2\mu_3... &\textrm{ if }\mu=\mu_1\mu_2...\in \bigcup_{N=2}^{\infty}E^N_{sinks}\cup E^\infty\\
r(\mu) & \textrm{ if }\mu\in E^1_{sinks}.
\end{cases}
$$
  There is a natural   `product'  topology on $\Omega_E$ with the basis formed by the cylinder sets  
$U_{\nu}=\{\nu\mu : \nu\mu \in \Omega_E\}$, $\nu\in  E^n, \,n\in \N$.
Equipped with this topology 
$\Omega_E$ is a compact Hausdorff space and  $\sigma_E$ is a local homeomorphism whose both  domain and codomain are clopen. Moreover, the standard argument, cf. e.g. \cite[Lemma 3.2]{jp}, shows that $s_\nu s_\nu^*\mapsto \chi_{U_\nu}$, $\nu\in  E^n, \,n\in \N$, establishes an isomorphism $\D_E\cong C(\Omega_E)$  which intertwines   $\phi_E:\D_E\to \D_E$ with the operator of composition with $\sigma_E$.

Let us consider the relation of `eventual equality' defined on $\Omega_E$ as follows:  
 $$
 \mu  \sim \nu \,\,\,\, \stackrel{def}{\Longleftrightarrow} \,\,\,\, 
 \begin{cases} 
  \nu,\mu\in E^\infty \textrm{ and  } \mu_N\mu_{N+1}...= \nu_N\nu_{N+1}... \,\textrm{ for some }\, N\in \N,
  \\
 \nu,\mu\in E^N_{sinks}\, \textrm{ for some }\, N\in \N\,\textrm{ and }\, r(\mu_N)=r(\nu_N).
  \end{cases}
 $$   
 Plainly, $\sim$ is an equivalence relation. We denote by $[\mu]$ the  equivalence class of $\mu\in \Omega_E$, and view  $\Omega_E/\sim$ as a topological space  equipped with the quotient topology. 
 \begin{lem}\label{lemma on quotient space}
The quotient map $q:\Omega_E \mapsto \Omega_E/\sim$ is open and the sets  
\begin{equation} \label{basis for quotients}  
U_{v,n}:=\{[\mu]: \exists_{\eta \in E^{k}, k\in \N} \,\, s(\eta)=v, r(\eta)=\mu_{n+k}\}, \quad v\in r(E^n), n\in \N,
\end{equation}
form a basis for the quotient topology of $\Omega_E/\sim$. Moreover, the formula 
 \begin{equation}\label{definition of partial homeomorphism dual to graph}
 [\sigma_E][\mu]:=[\sigma_E(\mu)]
 \end{equation}
defines a partial homeomorphism of $\Omega_E/\sim$ with natural domain and codomain:  
\begin{align*}\{[\mu]:\mu \in \Omega_E\setminus E^0_{sinks}\}&=\bigcup_{v\in E^0\setminus E^0_{sinks}} U_{v,0},
\\
\{[\mu] \in \Omega_E: s(\mu) \notin E^0_{sources}\}&=\bigcup_{v\in E^0\setminus E^0_{sources}} U_{v,0}.
\end{align*} \end{lem}
\begin{proof}
A moment of thought yields   that if $\nu \in E^n$ is such that $r(\nu)=v$, then $q(U_{\nu})=U_{v,n}$. In particular, one sees that 
\begin{align*}
q^{-1}(U_{v,n})
&=\{\mu \in \Omega_E:  \exists_{\eta \in E^{k}, k\in \N} \,\, s(\eta)=v, r(\eta)=\mu_{n+k} \}
\\
&= \bigcup_{k\in \N} \bigcup_{\eta\in E^{k},  \atop s(\eta)=v} \bigcup_{\nu \in E^{n+k} \atop r(\nu)=r(\eta)}\,\,\, U_{\nu},
\end{align*}
which means that   $U_{v,n}$ is open in $\Omega_E/\sim$. We conclude that \eqref{basis for quotients} defines a basis for the topology of $\Omega_E/\sim$ and $q$ is an open map.
\\
Now, it is straightforward to check that  \eqref{definition of partial homeomorphism dual to graph} gives a well defined mapping whose domain and codomain are open sets of the form described in the assertion. The map $[\sigma_E]$  is invertible as for $\mu \in \Omega_E$  such that $s(\mu) \notin E^0_{sources}$ its inverse can be described by the formula
$$
[\sigma_E]^{-1}[\mu]=[e\mu]\,\,\,\, \textrm{ for an arbitrary edge } e\in E^1 \textrm{ such that } r(e)=s(\mu),
$$
where $e\mu:=e$ when $\mu\in E^0_{sinks}$ is a vertex. Since $[\sigma_E](U_{v,n+1})=U_{v,n}$  and  $[\sigma_E]^{-1}(U_{v,n})=U_{v,n+1}$ for $v\in E^n$, $n \in \N$, we see that $[\sigma_{E}]$ is a partial homeomorphism.
\end{proof}

 We  show that the quotient partial reversible dynamical system $(\Omega_E/\sim, [\sigma_E])$  embeds as a dense subsystem into $(\widehat{\FF_E}, \widehat{\VV})$. Under this embedding the relation $\sim$ coincides with the unitary equivalence of GNS-representations associated to pure extensions of the pure states of $\D_E=C(\Omega_E)$. 
More precisely, for any  path $\mu\in \Omega_E$ the formula
\begin{equation}\label{product states equation}
\omega_\mu(s_\nu s_\eta^*)=
\begin{cases}
1 & \nu=\eta=\mu_1...\mu_n \\
0 & \textrm{otherwise} 
\end{cases},\qquad \textrm{for}\quad \nu,\eta\in E^n, \,\, n\in \N,
\end{equation}
determines  a pure state $\omega_\mu:\FF_E \to \C$  (a pure  extension of the point evaluation $\delta_\mu$ acting on  $\D_E=C(\Omega_E)$). Indeed, the functional $\omega_\mu$ is a pure state on each $\FF_k$, $k\in \N$,
and thus it is also a pure state on $\FF_E=\overline{\bigcup_{k\in \N}\FF_k}$, cf. e.g. \cite[4.16]{Bratteli}. We denote  by $\pi_\mu$ the GNS-representation associated to  $\omega_\mu$ and take up the rest of the  subsection to prove the following

 \begin{thm}[Partial homeomorphism dual to a graph interaction]\label{shifts of path representations}
 Under the above notation 
 $
 [\mu] \mapsto \pi_\mu 
 $
 is a topological embedding of $\Omega_E/\sim$  as a dense subset  into $\widehat{\FF_E}$.  This embedding intertwines $[\sigma_E]$ and $\widehat{\VV}$. Accordingly,
 the space $\widehat{\FF_E}$ admits the following decomposition into disjoint sets
$$
\widehat{\FF_E}=\bigcup_{N=0}^{\infty}\widehat{G}_N \cup \widehat{G}_\infty
$$
where the sets $\widehat{G}_N=\{\pi_\mu: [\mu]\in E^N_{sinks}/\sim\}$ are  open discrete  and  $\widehat{G}_\infty=\overline{\{\pi_\mu: [\mu] \in E^\infty/\sim\}}$ is a closed subset of $\widehat{\FF_E}$.  The set
$$
\Delta=\widehat{\FF_E}\setminus \widehat{G}_0
$$
is the
domain of $\widehat{\VV}$, and  $\widehat{\VV}$ is uniquely determined by the formula
$$
\widehat{\VV}(\pi_{\mu})=\pi_{\sigma_E(\mu)},\qquad      \mu \in \Omega_E\setminus E^0_{sinks}.
$$
 In particular, $
 \pi_{\mu} \in \widehat{\VV}(\Delta)$, for $\mu\in E^{N}_{sinks}$, iff  there is $\nu\in E^{N+1}_{sinks}$  such that  $r(\mu)=r(\nu)$,
 and  then $\widehat{\HH}(\pi_{\mu})=\pi_{\nu}$. Similarly,
   $   \pi_{\mu} \in \widehat{\VV}(\Delta)$, for $\mu\in E^{\infty}$,  iff there is $\nu\sim \mu$ such that  $s(\nu)$ is not a source, and then for any
   $\nu_0 \in E^{1}$ such that $\nu_0\nu_1\nu_2...\in E^{\infty}$ we have  
  $\widehat{\HH}(\pi_{\mu})=\pi_{\nu_0\nu_1\nu_2,...}
 $.
 \end{thm}
\begin{rem}\label{dilation remark}
One may verify that if we put 
$$A:=\spane(\{p_v: v\in E^0_{sinks}\}\cup \{s_es_e^*: e\in E^1\})\cong\C^{|E^0_{sinks}|+|E^1|},
$$
 then $\HH$ preserves $A$ and the smallest $C^*$-algebra containing $A$ and invariant under $\VV$ is  $\FF_E$. In this sense $\HH:\FF_E \to \FF_E$  is a natural  dilation of the positive linear map $\HH:A\to A$.
  This explains the  similarity of assertions in Theorem \ref{shifts of path representations} and in \cite[Theorem 3.5]{kwa-logist}; both of these results describe dual partial homeomorphisms obtained in the process of dilations. The essential difference is that a  dilation of a multiplicative map on a commutative algebra always leads a commutative $C^*$-algebra, cf. \cite{kwa-ext}, \cite{kwa-logist}, while a stochastic factor manifested by a lack of multiplicativity of the initial mapping inevitably leads to noncommutative objects after a dilation. Significantly, our dual picture of the graph interaction  `collapses' to the non Hausdorff quotient similar to that of  Penrose tilings \cite[3.2]{Connes}.
\end{rem}

We start by noting that the infinite direct sum $ \oplus_{N=0}^\infty  \oplus_{\,w\in E^0_{sinks}} \FF_N(w)$ yields an ideal $I_{sinks}$ in $\FF_E$  generated by the projections  $p_w$, $w\in  E^0_{sinks}$. 
We rewrite it in the form
$$
I_{sinks}=\bigoplus_{N\in \N} G_N, \qquad \quad\textrm{where }\quad G_N:=\Big(\oplus_{w\in E^0_{sinks}}   \FF_N(w)\Big).
$$
Plainly, $\FF_N(w)\neq \{0\}$ for  $w\in E^0_{sinks}$  iff  there is $\mu\in E^N_{sinks}$ such that $r(\mu)=w$ and then (since $\FF_N(w)$  is a finite factor)  $\pi_{\mu}$ is a  unique up to unitary equivalence irreducible representation  of $\FF_E$ such that 
$
\ker\pi_{\mu}\cap \FF_N(w)=\{0\}. 
$
Consequently, we see that 
$$  
  \widehat{I}_{sinks}=\bigcup_{N=0}^{\infty}\widehat{G}_N  \ni \pi_\mu   \longmapsto [\mu] \in \bigcup_{N=0}^{\infty}E^N_{sinks}/\sim 
$$ establishes a homeomorphism 
between the corresponding discrete spaces.
The complement of $\widehat{I}_{sinks}=\bigcup_{N=0}^{\infty}\widehat{G}_N$ in $\widehat{\FF_E}$ is a closed set  which we identify in a usual way with  the spectrum of the   quotient algebra 
$$
G_\infty:= \FF_E/I_{sinks}.
$$ 
We will describe a  dense subset of   $\widehat{G}_\infty$ exploiting the fact that states $\omega_\mu$ arising from $\mu \in  E^{\infty}$ can be considered as analogs of Glimm's product states for UHF-algebras, cf. e.g. \cite[6.5]{Pedersen}.
\begin{lem}\label{path representations}
For infinite paths $\mu,\nu\in E^\infty$  the  representations $\pi_\mu$ and $\pi_\nu$  are unitarily equivalent if and only if  $\mu \sim \nu$.  In particular, $
 [\mu] \mapsto \pi_\mu 
 $
 is a well defined embedding of $E^\infty/\sim$ into $\widehat{G}_\infty$. 
\end{lem}
\begin{proof}
We mimic the proof of the corresponding result  for UHF-algebras, cf. \cite[6.5.6]{Pedersen}.
 Note that if  $(\mu_{N+1},\mu_{N+2},...)= (\nu_{N+1},\nu_{N+2},...)$, then both $s_{\mu_1...\mu_N} s_{\mu_1...\mu_N}^*$ and $s_{\nu_1,...,\nu_N} s_{\nu_1,...,\nu_N}^*$ are in $F_N(v)$ where $v=r(\mu_N)$ and since $F_N(v)\cong M_n(\C)$ there is a unitary $u\in F_N(v)$ such that  $\omega_\mu(a)=\omega_{\nu}(u^*au)$ for $a\in F_N(v)$. Then automatically   $\omega_\mu(a)=\omega_{\nu}(u^*au)$ for all $a\in \FF_E$ and hence $\pi_\mu\cong\pi_\nu$.  Conversely, suppose that $\pi_\mu\cong\pi_\nu$, then, cf. \cite[3.13.4]{Pedersen}, there is a unitary $u\in \FF_E$ such that  $\omega_\mu(a)=\omega_{\nu}(u^*au)$ for all $a\in \FF_E$. For sufficiently large $n$ there is $x\in \FF_n$ with $\|u-x\|<\frac{1}{2}$ and $\|x\|\leq 1$. To get  the contradiction we assume that  $\mu_k\neq \nu_k$ for some $k>n$. The element
$
a:= s_{\mu_{1}...\mu_{k}}s_{\mu_{1}...\mu_{k}}^*\in \FF_{k}
$
commutes with all the elements from $\FF_n$. Indeed,  if $b=s_{\al}s_{\beta}^*\in \FF_{n}$,  $\al,\beta \in E^n$, then either $s_{\al f}s_{\beta f}^*=0$ for all $f\in E^{k-n}$ and then $ab=ba=0$ or $b=\sum_{f\in E^{k-n}} s_{\al f} s_{\beta f}^*$ and then 
$
ab=  s_{\al\mu_{n+1}...\mu_{k}} s_{\beta\mu_{n+1}...\mu_{k}}^*=ba.
$ 
From this it also follows that $\omega_\nu(ab)=0$ for all $b\in \FF_n$. Accordingly, $\omega_\nu(x^* ax)=\omega_\nu( ax^*x)=0$ and since $\|(u^*-x^*) au\|=\|x^* a(u-x)\|<1/2$ we get
$$
1> \omega_\nu((u^*-x^*) au) + \omega_\nu(x^* a(u-x)) = \omega_\nu(u^* au) - \omega_\nu(x^* ax)=\omega_\mu(a)=1,
$$
 an absurd.
\end{proof}
 \begin{rem}\label{remark on sinkless} The $C^*$-algebra $G_\infty$ is a  graph algebra arising from a  graph which  has no sinks. Indeed,  the saturation $\overline{E^0}_{sinks}$  of 
  $E^0_{sinks}$ (the minimal saturated set containing $ E^0_{sinks}$) is the hereditary and saturated set corresponding to  the ideal $I_{sinks}$ in $\FF_E$. Hence   $I_{sinks}=\FF_{\overline{E^0}_{sinks}}$ and
$$
G_\infty\cong \FF_{E^0_{sinkless}} \quad\textrm{ where } \quad E^0_{sinkless}:=E^0\setminus\overline{E^0}_{sinks}.
$$
\end{rem}

Let us now treat $\mu\in E^\infty$ as the full subdiagram of the Bratelli diagram $\Lambda(E)$ where  the only vertex on the $n$-th level is $r(\mu_n)$. Similarly, we treat $\mu\in E^N_{sinks}$ as the full subdiagram of $\Lambda(E)$ where on the $n$-th level for $n\leq N$ is $r(\mu_n)$   and for  $n>N$ is $r(\mu)^{(N)}$, cf. notation in subsection \ref{graph core subsection}. For any $\mu \in \Omega_E$ we  denote by $W(\mu)$  the full subdiagram of  $\Lambda(E)$  consisting of  all ancestors of the base vertices of $\mu\subset \Lambda(E)$.
\begin{lem}\label{lemma on diagrams for kernels}
For any $\mu \in \Omega_E$ the Bratteli subdiagram $\Lambda(\ker\pi_\mu)$ of $\Lambda(E)$ corresponding to  $\ker\pi_\mu$ is  $\Lambda(E)\setminus W(\mu)$.
\end{lem}
\begin{proof} The assertion follows immediately from the form of primitive ideal subdiagrams, see \cite[3.8]{Bratteli}, definition \eqref{product states equation} of $\omega_\mu$ and the fact that $\ker\pi_\mu$ is the largest ideal contained in $\ker\omega_\mu$.
\end{proof}
\begin{lem}
The mapping $
 [\mu] \mapsto \pi_\mu 
\in \widehat{\FF_E} $
 is a homeomorphism from $\Omega_E/\sim$ onto its image.\end{lem}
\begin{proof}
We already know that $
 [\mu] \mapsto \pi_\mu 
 $ is injective and restricts to homeomorphism between discrete spaces $(\Omega_E\setminus E^{\infty})/\sim$  and $\widehat{\FF_E}\setminus \widehat{G}_\infty$.  Hence it suffices to prove that $
 [\mu] \mapsto \pi_\mu 
 $ is continuous and open when considered as a mapping from $E^\infty/\sim$ onto $\{\pi_\mu: [\mu]\in E^\infty/\sim\}\subset \widehat{G}_\infty$.
 To this end, we may assume that $E$ has no sinks, cf. Remark \ref{remark on sinkless}. Suppose then that $E$ has no sinks.
 
   Any  open set in $\widehat{\FF_E}$ is of the form 
$
\widehat{J}=\{\pi \in \widehat{\FF_E}:  \ker \pi \nsupseteq J\}=\{\pi \in \widehat{\FF_E}:  \Lambda(J)\setminus \Lambda(\ker\pi) \neq \emptyset\}
$
where $J$ is an  ideal in $\FF_E$ or equivalently $\Lambda(J)$ is a hereditary and saturated subdiagram of $\Lambda(E)$. It follows that if we denote by $\Lambda_{v,n}$ the smallest hereditary and saturated subdiagram of $\Lambda(E)$ which  on the $n$-th level contains vertex $v$,  then the sets
$$
\widehat{J}_{v,n}:=\{\pi \in \widehat{\FF_E}: \Lambda_{v,n}\setminus \Lambda(\ker\pi) \neq \emptyset\}, \qquad v\in E^0, n\in \N, 
$$
form a basis for the topology of $\widehat{\FF_E}$. Moreover, in view of Lemma \ref{lemma on diagrams for kernels}, definitions of $\Lambda_{v,n}$, $W(\mu)$ and form of 
$U_{v,n}$, see \eqref{basis for quotients}, the preimage of $
\widehat{J}_{v,n}$ under the map $
 [\mu] \mapsto \pi_\mu 
 $ is
 $\{[\mu] \in \Omega_E/\sim:  \Lambda_{v,n} \cap W(\mu)\neq \emptyset\}
 = \{[\mu] \in \Omega_E: \exists_{\nu \in E^{k}, k\in \N} \,\, s(\nu)=v, r(\nu)=\mu_{n+k}\}
=U_{v,n}$.

 Thus, in view of Lemma \ref{lemma on quotient space}, we see that  $
 [\mu] \mapsto \pi_\mu 
 $ establishes one-to-one correspondence between the topological bases for its domain and codomain and hence is a  homeomorphism onto codomain.
\end{proof}
Now, to obtain Theorem \ref{shifts of path representations}  we only need the following
\begin{lem}
The mapping $
 [\mu] \mapsto \pi_\mu 
\in \widehat{\FF_E} $
 intertwines $[\sigma_E]$ and $\widehat{\VV}$.\end{lem}
\begin{proof}

 To see that  $\widehat{\VV(\FF_E)}=\{\pi \in \widehat{\FF_E}: \pi(\VV(1))\neq 0\}$  coincides with $\Delta=\widehat{\FF_E}\setminus \widehat{G}_0$ let $\pi \in \widehat{\FF_E}$ and note that  
  $$
 \pi(\VV(1))= 0 \Longleftrightarrow \forall_{v \in s(E^1)}\, \pi(p_v)=0  \Longleftrightarrow \exists_{w \in E^0_{sinks}} \pi\cong \pi_{w}.
 $$ 
 Furthermore, by \eqref{endomorphism action on F_A} and \eqref{transfer action on F_A},  
 we have
\begin{equation}\label{formulas for shifts on diagrams}
\VV(\FF_N(v))=\VV(1) \FF_{N+1}(v) \VV(1),\quad \HH(\FF_{N+1}(v))=\FF_{N}(v), \quad N\in\N,
\end{equation}
 and  $\HH(\FF_0(v))\subset \sum\limits_{w\in r(s^{-1}(v))} \FF_{0}(w)$.
In particular, for $\mu\in E^N_{sinks}$, $N>0$,  we have  $\pi_{\mu}\in \Delta$   and  
$$
(\pi_{\mu}\circ \VV)(\FF_{N-1}(w))=\pi_{\mu}(\VV(1)\FF_{N}(w)\VV(1))\neq 0.
$$
Hence $\widehat{\VV}(\pi_{\mu}) \cong \pi_{\sigma_E(\mu)}$.
Let us now fix $\mu=\mu_1\mu_{2}\mu_3...\in E^\infty$. Let   $H_\mu$ be the Hilbert space  and 
 $\xi_\mu\in H_\mu$  the cyclic vector associated to the pure state $\omega_\mu$ via GNS-construction.
For  $\nu,\eta\in E^n$, using  \eqref{endomorphism action on F_A} and  \eqref{product states equation},  we get
\begin{align*}
\omega_\mu(\VV(s_\nu s_\eta^*))&=\begin{cases} \frac{1}{ \sqrt{n_{s(\nu)}n_{s(\eta)}}} \sum\limits_{e,f\in E^1}\omega_\mu( s_{e\nu} s^*_{f\eta}), &  n_{s(\nu)}n_{s(\eta)}\neq 0, \\
 0 , &  n_{s(\nu)}n_{s(\eta)}= 0, 
 \end{cases}
 \\
&=\begin{cases}  \frac{1}{ n_{r(\mu_1)}} ,   & \nu=\eta=\mu_2...\mu_{n+1} \\
0 & \textrm{otherwise} 
 \end{cases}
 \\
 &=\frac{1}{n_{r(\mu_1)}}\, \omega_{\sigma_E(\mu)}(s_\nu s_\eta^*).
\end{align*}
It follows that $
 \omega_\mu\circ \VV = \frac{1}{n_{r(\mu_1)}}\, \omega_{\sigma_E(\mu)}
 $ and  therefore $\widehat{\VV}(\pi_\mu)\cong \pi_{\sigma_E(\mu)}$, cf.  \cite[Corollary  3.3.8]{Pedersen}. $\qquad$ $\qquad$
 \end{proof}

 \subsection{Topological freeness of graph interactions.}\label{final subsection}
We will now use Theorem \ref{shifts of path representations} to identify the relevant properties of the partial homeomorphism $\widehat{\VV}$ dual to the graph interaction $(\VV,\HH)$. 
We recall that the condition (L) introduced in \cite{kum-pask-rae}  requires that every loop in $E$ has an exit. For convenience, by loops we will mean simple loops, that is paths $\mu=\mu_1...\mu_n$ such that $s(\mu_1)=r(\mu_n)$ and $s(\mu_1)\neq r(\mu_k)$ for $k=1,...,n-1$. A loop $\mu$ is said to have an exit  if  there is an edge $e$ such that $s(e)=s(\mu_i)$ and $e\neq \mu_i$ for some $i=1,...,n$.

\begin{prop}\label{every loop  has an exit}
Suppose that every loop in $E$ has an exit. Then every  open set  intersecting $\widehat{G}_\infty$ contains infinitely many  non-periodic points for $\widehat{\VV}$ and if $E$ has no sinks the number of this non-periodic points is uncountable. In particular, $\widehat{\VV}$ is topologically free.
\end{prop}
\begin{proof} By Theorem \ref{shifts of path representations} and Lemma \ref{lemma on quotient space} it suffices to consider the dynamical system $(\Omega_E/\sim,[\sigma_E])$ and an open set of the form
$$
U_{n,v}=\{[\mu]: \exists_{\eta \in E^{k}, k\in \N} \,\, s(\eta)=v, r(\eta)=\mu_{n+k}\}
$$
which contains $[\mu]$ for $\mu=\mu_1\mu_2... \in E^\infty$. 
 Since  $E$ is finite there must be  a vertex $v$ which appears as a base point of $\mu$ infinitely many times. Namely, there exists an increasing  sequence $\{n_k\}_{k\in \N}\subset \N$ such that $r(\mu_{n_k})=v$ for all $k\in \N$. Moreover, since every loop in $E$ has an exit, the vertex $v$ has to be connected either to a sink or to a vertex lying on two different loops. Let us consider these two cases: 

1)  Suppose $\nu$ is a finite path such that  $v=s(\nu)$ and $w:=r(\nu)\in E^{0}_{sinks}$.  Consider the family of finite, and hence non-periodic, paths 
$$
\mu^{(n_k)}:=\mu_1...\mu_{n_k}\nu \in E_{sinks}^{n+|\nu|}, \qquad k\in \N.
$$
Plainly, all  except finitely many of  elements $[\mu^{(n_k)}]$ belong to  $U_{n,v}$ (and they are all different).

2) Suppose  $\nu$ is a finite path such that  $v=s(\nu)$ and the vertex $w:=r(\nu)$ is a base point for two different loops $\mu^0$ and $\mu^1$.
 We put  $\mu^\epsilon=\mu^{\epsilon_1}\mu^{\epsilon_2}\mu^{\epsilon_3}...\in E^\infty$ for $\epsilon=\{\epsilon_i\}_{i=1}^\infty\in \{0,1\}^{\N\setminus \{0\}}$. 
   Since there is an uncountable number of non-periodic sequences in $\{0,1\}^{\N\setminus \{0\}}$ which pair-wisely do not eventually  coincide the paths $\mu^\epsilon$ corresponding to these sequences give rise to the uncountable family of  non-periodic elements $[\mu^\epsilon]$ in $\Omega_E/\sim$. Moreover, one readily sees that for sufficiently large $n_k$ all the equivalent classes of paths
   $$
   \mu^{(\epsilon)}:=\mu_1...\mu_{n_k}\nu\mu^\epsilon \in E^{\infty}, \qquad \epsilon=\{\epsilon_i\}_{i=1}^\infty\in \{0,1\}^{\N\setminus \{0\}}
   $$
   belong to $U_{n,v}$. This proves our assertion.
\end{proof}
\begin{ex}\label{remark on topological freeness}
In the case $C^*(E)=\OO_n$  is the Cuntz algebra, that is when  $E$ is the graph with a single vertex and $n$ edges, $n\geq 2$, then   $\FF_E$ is an UHF-algebra and the states $\omega_\mu$ are simply Glimm's product states. In particular, it is well known that   $\Prim(\FF_E)=\{0\}$ and $\widehat{\FF_E}$ is uncountable, cf.  \cite[6.5.6]{Pedersen}. Hence, on one hand  the Rieffel homeomorphism  given by the imprimitivity $\FF_E$-bimodule $X=\FF_E s \FF_E$ associated with the graph interaction $(\VV,\HH)$ is the identity on $\Prim(\FF_E)$, and thereby it is not topologically free (\cite[Theorem 6.5]{kwa-demon} can not be applied).  On the other hand, we have just shown that  $\widehat{\FF_E}$ contains  uncountably many non-periodic points for $X\dashind=\widehat{\VV}^{-1}$, cf. Proposition \ref{coincidence of actions}, and hence it
 is  topologically free.
\end{ex}

Suppose now that  $\mu$ is a loop in $E$. Let $\mu^\infty \in E^\infty$ be the path obtained by the infinite concatenation of $\mu$.   Then   $\Lambda(E)\setminus W(\mu^\infty)$ is a Bratteli diagram for a primitive ideal in $\FF_E$, which we  denote by  $I_\mu$. In other words, see Lemma \ref{lemma on diagrams for kernels}, we have
$$
I_\mu=\ker\pi_{\mu^\infty}
$$ 
 where $\pi_{\mu^\infty}$ is the irreducible representation associated to $\mu^\infty$.
\begin{prop}\label{no exit no freedom}
 If the loop $\mu$  has no exits, then up to unitary equivalence $\pi_{\mu^\infty}$ is the only     representation  of $\FF_E$ whose kernel  is
$I_\mu$ and  the singleton $\{\pi_{\mu^\infty}\}$ is an open set in  $\widehat{\FF_E}$. 
\end{prop}
\begin{proof}
 The quotient $\FF_E/ I_\mu$ is an AF-algebra with the  diagram $W(\mu^\infty)$. 
  The path $\mu^\infty$ treated as a subdiagram of $W(\mu^\infty)$ is  hereditary  and its saturation $\overline{\mu^\infty}$ yields an ideal $\K$ in  $\FF_E/ I_\mu$. Since $\mu^\infty$ has no exits,  $\K$ is isomorphic to the ideal of compact operators $\K(H)$ on a separable Hilbert space $H$  (finite or infinite dimensional).  
 Therefore every faithful irreducible representation of  $\FF_E/ I_\mu$ is unitarily equivalent to the representation given by the isomorphism $\K\cong\K(H)\subset \B(H)$. This shows that $\pi_{\mu^\infty}$ is determined by its kernel. Moreover, since $W(\mu^\infty)$ contains all its ancestors, the subdiagram    $\overline{\mu^\infty}$  is  hereditary and saturated not only in $W(\mu^\infty)$ but also in $\Lambda(E)$. Therefore   we let now $\K$ stand for  the ideal in $\FF_E$, corresponding to  $\overline{\mu^\infty}$. Let $P\in \Prim(\FF_E)$.  As $\K$ is simple   $P \nsupseteq \K$ implies $\K\cap P=\{0\}$. By the form of $W(\mu^\infty)$ and  hereditariness of $\Lambda(P)$, $\K\cap P=\{0\}$ implies $\Lambda(P)\subset \Lambda(\FF_E)\setminus W(\mu^\infty)=\Lambda(I_\mu)$. However, if  $P\subset I_\mu$, we must have $P=I_\mu$ because no part of $\Lambda(I_\mu)$ is not connected to $W(\mu^\infty)$ (consult the form of diagrams of primitive ideals  \cite[3.8]{Bratteli}). Concluding,   we get  
$$
\{P \in \Prim(\FF_E): P \nsupseteq \K\} =\{P \in \Prim(\FF_E): \K\cap P=\{0\}\} = \{I_\mu\},  
$$
 which means that $\{I_\mu\}$ is open in $\Prim(\FF_E)$. Accordingly,   $\{\pi_{\mu^\infty}\}$ is  open  in  $\widehat{\FF_E}$.
\end{proof}
We have the following characterization of  minimality of $\widehat{\VV}$.

\begin{prop}\label{invariant Bratteli diagrams}
The map 
$
V \mapsto \widehat{\FF_{\Lambda(V)}}
$
is a one-to-one correspondence between the hereditary saturated subsets of $E^0$  and  $\widehat{\VV}$-invariant open subsets of $\widehat{\FF_E}$. In particular,
  $\widehat{\VV}$ is minimal   if and only if there are no nontrivial hereditary saturated subsets of $E^0$.
\end{prop}
\begin{proof} Recall that for a hereditary and saturated subset $V$ of $E^0$ we treat $\Lambda(V)$ as a subdiagram of $\Lambda(E)$ where on the $n$-th level we have  $(r(E^n)\cap V) \cup \bigcup_{k=0}^{N-1}\{v^{(k)}: v \in r(E^{k}_{sinks})\cap V\}$. Now, using  condition iii)  of  Lemma \ref{invariant ideals for endomorophisms} and relations \eqref{formulas for shifts on diagrams} one readily see that the open set   $\widehat{I}$ for an ideal $I$  in $\FF_E$ is $\widehat{\VV}$-invariant if and only if the corresponding  Bratteli diagram for $I$ is of the form $\Lambda(V)$ where $V\subset E^0$ is  hereditary and saturated.
\end{proof}

Combining the above results we  not only characterize  freeness and topological freeness of $(\widehat{\FF_E},\widehat{\VV})$ but also spot out an interesting  dichotomy concerning its core subsystem $(\widehat{G}_\infty, \widehat{\VV})$, cf. Remark \ref{pure infiniteness remark} below. 
\begin{thm}\label{dynamical dychotomy}
Let $(\widehat{\FF_E},\widehat{\VV})$ be a partial homeomorphism dual to the graph interaction $(\VV,\HH)$. We have the following  dynamical dichotomy: 
\begin{itemize}
\item[a)] either every  open set intersecting  $\widehat{G}_\infty$ contains infinitely many  nonperiodic points  for $\widehat{\VV}$; this holds if every loop in  $E$ has an exit,  or
\item[b)] there are $\widehat{\VV}$-periodic orbits $O=\{\pi_{\mu^\infty},\pi_{\sigma_E(\mu^\infty)}...,\pi_{\sigma_E^{n-1}(\mu^\infty)}\}$  in $\widehat{G}_\infty$ forming  open discrete sets in  $\widehat{\FF_E}$; they correspond to loops  without exits $\mu$.
\end{itemize}
In particular,  
\begin{itemize}
\item[I)] $\widehat{\VV}$ is topologically free if and only if every loop in $E$ has an exit (satisfies condition (L)),
\item[II)] $\widehat{\VV}$ is free if and only if every loop has an exit connected to this loop (satisfies the so-called condition (K) introduced in \cite{kum-pask-rae-ren}, see also \cite{bprs}).
\end{itemize}
\end{thm}
\begin{proof}
In view of Propositions \ref{every loop  has an exit}, \ref{no exit no freedom} only  item II) requires a comment. By Proposition \ref{invariant Bratteli diagrams} every closed $\widehat{\VV}$-invariant set is of the form $\widehat{\FF_E}\setminus \widehat{\FF_V}=\widehat{\FF_{E\setminus V}}$  for a hereditary and saturated subset $V\subset E^0$. Hence   $\widehat{\VV}$ is free if and only if every loop outside a hereditary saturated set $V$ has an  exit outside $V$. The latter condition  is clearly equivalent to the condition that every loop has an exit connected to this loop,   cf. \cite[page 318]{bprs}. 
\end{proof}
\begin{rem}\label{pure infiniteness remark}
Since  $E$ is finite, by \cite[Theorem 3.9]{kum-pask-rae},   $C^*(E)$ is purely infinite in the sens of \cite{Lac-Spiel}, \cite{kum-pask-rae} if and only if $E$ has no sinks and every loop in $E$ has an exit.  In view of Proposition \ref{every loop  has an exit} we conclude  that  $C^*(E)$ is purely infinite if and only if every nonempty open set in  $\widehat{\FF_E}$ contains uncountable number of  nonperiodic points  for $\widehat{\VV}$. In particular, every $\widehat{\VV}$-periodic orbit $O=\{\pi_{\mu^\infty},\pi_{\sigma_E(\mu^\infty)}...,\pi_{\sigma_E^{n-1}(\mu^\infty)}\}$ yields a  gauge invariant ideal $J_O$ in $C^*(E)$ (generated by  $\bigcap_{\pi \in \widehat{\FF_E}\setminus O} \ker\pi$) which is not purely infinite.  Indeed, if $v=s(\mu)$ is the source of a loop $\mu$ witch has no exit, then 
$
p_v C^*(E)p_v=p_v J_O p_v = C^*(s_\mu)\cong C(\T)
$
because $s_\mu$ is a unitary in $C^*(s_\mu)$ with the full spectrum, cf. \cite[proof of Theorem 2.4]{kum-pask-rae}.
\end{rem}

Concluding,  we deduce from our   general results for corner interactions  the following fundamental classic results  for graph algebras, cf. \cite{Raeburn}, \cite{kum-pask-rae}, \cite{kum-pask-rae-ren},  \cite{bprs}.
\begin{cor}\label{graphs results}
 Consider the graph $C^*$-algebra $C^*(E)$ of the finite directed graph $E$.
\begin{itemize}
\item[i)] If every loop in $E$ has an exit, then any  Cuntz-Krieger $E$-family $\{P_v: v\in E^0\}$, $\{S_{e}:e\in E^1\}$ generates a $C^*$-algebra isomorphic to $C^*(E)$, via $s_e\mapsto S_e$, $p_v\mapsto P_v$, $e\in E^1$, $v\in E^0$. 
\item[ii)] If every loop in $E$ has an exit connected to this loop, then there is a lattice isomorphism between hereditary saturated subsets of $E^0$ and ideals in $C^*(E)$, given by $V \mapsto J_V$, where $J_V$ is an ideal generated by $p_v$, $v\in V$.
 \item[iii)] If  every loop in $E$ has an exit and $E$ has no nontrivial hereditary saturated sets, then $C^*(E)$ is simple.
\end{itemize}
\end{cor}
\begin{proof} Apply  Propositions \ref{thm for Huef and Raeburn},  \ref{invariant Bratteli diagrams} and Theorems  \ref{interactions}, \ref{dynamical dychotomy}.
\end{proof}

\subsection{$K$-theory}
We now turn to description of $K$-groups for $C^*(E)$. As  $K_1$ groups for AF-algebras are trivial,  using Pimsner-Voiculescu sequence from Theorem \ref{Voicu-Pimsner for interacts} applied to the graph interaction $(\VV,\HH)$ associated to $E$  we have
 $$
  K_1(C^*(E) )\cong \ker(\iota_*- \HH_*),$$
  $$
    K_0(C^*(E) )\cong \coker(\iota_*- \HH_*)= K_0(\FF_E)/\im(\iota_*- \HH_*)   
  $$
  where $(\iota_*- \HH_*): \K_0(\VV(\FF_E))\to K_0(\FF_E)$. Hence to calculate the $K$-groups for $C^*(E)$ we need to identify $\ker(\iota_*- \HH_*)$ and $\coker(\iota_*- \HH_*)$. We do it in two steps.
 \begin{prop}[$K_0$-partial automorphism induced by a graph interaction]
 \label{label for K_0 automorphism}
   The group $K_0(\FF_E)$ is the universal abelian group $\langle V \rangle$ generated by  the set $V:=\{v^{(N)} : v\in r(E^N),\,\,  N\in \N\}$  of `endings of finite paths', subject to relations 
\begin{equation}\label{defining relations}
      v^{(N)}=\sum_{s(e)=v} r(e)^{(N+1)}\quad  \textrm{ for all } v \in r(E^N)\setminus E^0_{sinks}.
\end{equation}
In particular, the subgroup generated by  $v^{(N)}\in V$, $v\in E^0_{sinks}$, $N\in \N$, in $K_0(\FF_E)$ is  free abelian.
The groups  $K_0(\VV(A))$ and  $K_0(\HH(A))$ embeds into  subgroups of $K_0(\FF_E)$ and we have  
\begin{align*}
    K_0(\VV(\FF_E))&=\langle V\setminus \{v^{(0)}: v \in E_{sinks}^0\}\rangle,
   \\
    K_0(\HH(\FF_E))&=\langle \{ v^{(N)}\in V: v^{(N+1)} \in V\}\rangle.
\end{align*}
The isomorphism $\HH_*:K_0(\VV(A))\to K_0(\HH(A))$ is  determined by
\begin{equation}\label{K mapping H}
  \HH_*(v^{(N+1)})=v^{(N)},\qquad  \qquad N\in \N.
\end{equation}
   \end{prop}
   \begin{proof} We identify $v^{(N)}$ with the $K_0$-group element $[s_\mu s_\mu^*]$ where $\mu\in E^N$ and $v=r(\mu)$.
    It follows from  \eqref{inductive limit} that the group $K_0(\FF_E)$ is the inductive limit $\lim\limits_{\longrightarrow}(K_0(\FF_N),i_E^N)$ where 
   $$
  K_0(\FF_N)\cong \bigoplus_{v\in r(E^N)\setminus E^0_{sinks}} \Z v^{(N)} \oplus  \bigoplus_{k=0,...,N}\,\, \bigoplus_{ v\in r(E^k)\cap E^0_{sinks}} \Z v^{(k)}.
  $$ 
   Under the above isomorphisms, the bonding maps $i_E^N: K_0(\FF_N)\to  K_0(\FF_{N+1})$, $N\in \N$,  are given on generators by the formula
   $$
   i_E^N(v^{(N)})=\begin{cases} \sum_{s(e)=v} r(e)^{(N+1)}, & v \notin  E^0_{sinks} \\
   
   v^{(N)},  & v \in  E^0_{sinks}, 
      \end{cases}\qquad v \in r(E^N).
   $$
This immediately implies   the first part of the assertion.
\\
Since $\HH(\FF_E)=\HH(1)\FF_E\HH(1)$ is the closure of $\bigcup_{N\in \N} \HH(1)\FF_N\HH(1)$ and $K_0(\HH(1)\FF_N\HH(1))$ embeds into  $K_0(\FF_N)$ we see  by continuity pf $K_0$ that $K_0(\HH(\FF_E))$ embeds into $K_0(\FF_E)=\langle V \rangle$. Moreover, as  $\HH(1)=\sum_{v \in r(E^1)}p_v$ we get
   $$
   \HH(1)\FF_N(v)\HH(1)\neq \{0\} \,\, \Longleftrightarrow\,\, \FF_{N+1}(v)\neq \{0\} \,\, \Longleftrightarrow\,\, v^{(N+1)}\in V,
   $$
whence $K_0(\HH(\FF_E))$ identifies with $\langle \{ v^{(N)}\in V: v^{(N+1)} \in V\}\rangle$.
\\
   Similarly, since  $\VV(1)= \sum\limits_{v\in r(E^1)}\frac{1}{n_{v}} \sum\limits_{e,f\in r^{-1}(v)}\,\, s_es_f^*$ we infer that  $\VV(1)\FF_N(v)\VV(1)\neq \{0\}$ for all $v\in E^0$ and  $N>0$,  and $\VV(1)\FF_0(v)\VV(1)\neq \{0\}$ if and only if $v\notin E^0_{sinks}$. Thus we may identify $K_0(\VV(\FF_E))$ with $\langle V\setminus \{v^{(0)}: v \in E_{sinks}\}\rangle$. Now \eqref{K mapping H} follows from \eqref{transfer action on F_A}. Note that \eqref{K mapping H} determines $\HH_*$, as   for $v \in E^0\setminus E_{sinks}^0$, using only  \eqref{defining relations} and \eqref{K mapping H}  we have $\HH_*(v^{(0)})=\HH_*(\sum_{s(e)=v} r(e)^{(1)})=\sum_{s(e)=v} r(e)^{(0)}$.
   \end{proof}

 We  let $\Z (E^0\setminus E_{sinks}^0)$ and $\Z E^0$ denote the free abelian groups on free generators $E^0\setminus E_{sinks}^0$ and $E^0$, respectively. We consider the group homomorphism $\Delta_E:  \Z (E^0\setminus E_{sinks}^0) \to \Z E^0$ defined on generators as
  $$
  \Delta_E(v)= v- \sum_{s(e)=v} r(e).
  $$
 The following lemma can be viewed as a counterpart of Lemmas 3.3, 3.4 in \cite{Raeburn Szymanski}. Nevertheless,  it is a slightly different statement.
   \begin{lem}
    We have isomorphisms
    $$
    \ker(\iota_*- \HH_*)\cong \ker(\Delta_E), \qquad
            \coker(\iota_*- \HH_*)  \cong \coker(\Delta_E), 
    $$
which are given on generators by 
 \begin{align*}
   \ker \Delta_E \ni v  & \stackrel{i^{(0)}}{\longmapsto} v^{(0)} \in \ker(\iota_*- \HH_*)  , 
   \\
          \Z E^0/\im(\Delta_E)\ni  [v]   & \stackrel{j^{(0)}}{\longmapsto} [v^{(0)}] \in K_0(\FF_E)/\im(\iota_*- \HH_*) .
     \end{align*}
  \end{lem}
  \begin{proof} Suppose $a=\sum_{v\in E^0\setminus E_{sinks}^0} a_v v\in \Z(E^0\setminus E^0_{sinks})$. Then by \eqref{defining relations}, \eqref{K mapping H} we have 
\begin{align*}
(\iota_*- \HH_*)(i^{(0)}(a))&=\sum_{v\in E^0\setminus E_{sinks}^0} a_v v^{(0)}-\sum_{v\in E^0\setminus E_{sinks}^0} a_v  \sum_{s(e)=v} r(e)^{(0)}
  \\
  &=i^{(0)}(\Delta_E(a)).
\end{align*}
  Accordingly, $a\in \ker \Delta_E$ implies  $i^{(0)}(a)\in \ker(\iota_*- \HH_*)$ and hence  $i^{(0)}$ is well defined. Clearly $i^{(0)}$ is injective.
 To show that it is  surjective note that
\begin{equation}\label{general form of an element}
x=\sum_{v\in r(E^N)\setminus E_{sinks}^0} x_v v^{(N)}+ \sum_{k=1,...,N }\,\,\sum_{ v\in r(E^k)\cap E^0_{sinks}}  x_v^{(k)} v^{(k)} 
\end{equation}
is a general form of an element in $K_0(\VV(A))$ and assume $x$ is in $\ker(\iota_*- \HH_*)$. The  relation $x=\HH_*(x)$ implies that  the coefficients  corresponding to sinks in the expansion \eqref{general form of an element} are  zero. Thus   
$x=\HH_*^n(x)=\sum_{v\in E^0\setminus E_{sinks}^0} x_v v^{(0)}= i^{(0)}(a)$ where $a:=\sum_{v\in E^0\setminus E_{sinks}^0} x_v v$ is in $\ker \Delta_E $ because  $i^{(0)}(a)=x=\HH_*(x)=\HH_*(i^{(0)}(a))=i^{(0)}(\Delta_E(a))$. Hence $i^{(0)}$ is an isomorphism.
\\
Since $i^{(0)}$ intertwines $\Delta_E$ and $(\iota_*- \HH_*)$ we see that  $j^{(0)}$ is well defined.
 To show that $j^{(0)}$ is surjective, let $y=x + \sum_{v\in E_{sinks}^0} x_v^{(0)} v^{(0)}$ where $x$ is given by \eqref{general form of an element} (this is a general form of an element in $K_0(\FF_E)$).
Observe that as $x-\HH_*(x)\in \im(\iota_*- \HH_*)$ the element $y$ has the same class in $\coker(\iota_*- \HH_*)$ as 
$$
\HH_*(x)+  \sum_{v\in E_{sinks}^0} x_v^{(0)} v^{(0)}= z + \sum_{k=0,1} \,\, \sum_{v\in r(E^k)\cap E^0_{sinks}}  x_v^{(k)} v^{(0)}\,\,\, 
$$
where $
z=\sum_{v\in r(E^N)\setminus E_{sinks}^0} x_v v^{(N-1)}+ \sum_{k=2,...,N }\,\,\sum_{ v\in r(E^k)\cap E^0_{sinks}}  x_v^{(k)} v^{(k-1)} 
$ is in $ K_0(\VV(A))$. Applying the above argument to $z$ and proceeding in this way $N$ times we get that $y$ is in the same class in $\coker(\iota_*- \HH_*)$ as
$$
\sum_{v\in r(E^N)\setminus E_{sinks}^0} x_v v^{(0)}+ \sum_{k=0,1,...,N} \,\, \sum_{v\in r(E^k)\cap E^0_{sinks}} x_v^{(k)} v^{(0)}.
$$
Hence 
$$y=j^{(0)}\left(\sum_{v\in r(E^N)\setminus E_{sinks}^0} x_v [v]+ \sum_{k=0,1,...,N} \,\, \sum_{v\in r(E^k)\cap E^0_{sinks}} x_v^{(k)} [v] \right) 
.$$
The proof of injectivity of $j^{(0)}$ is slightly more complicated.  Let us consider $a=\sum_{v\in E^0} a_v v \in \Z E^0$  such that $i^{0}(a)\in \im(\iota_*- \HH_*)$. Then 
$i^{0}(a)=x- \HH_*(x)$ for an element $x$ of the form \eqref{general form of an element}, and hence  
\begin{align*}
i^{0}(a)=&\sum_{v\in r(E^N)\setminus E_{sinks}^0} x_v \left(v^{(N)}- \sum_{s(e)=v} r(e)^{(N)}\right)
\\
&+ \sum_{k=1,...,N} \,\, \sum_{v\in r(E^k)\cap E^0_{sinks}}x_v^{(k)} \left(v^{(k)}- v^{(k-1)}\right).
\end{align*}
On the other hand, applying  $N$-times relation \eqref{defining relations}  to $i^{0}(a)=\sum_{v\in E^0} a_v v^{(0)}$ we get 
$$
i^{0}(a)=\sum_{\mu\in E^N} a_{s(\mu)} v_{r(\mu)}^{(N)}+  \sum_{k=0,...,N} \,\, \sum_{\mu\in E^k_{sinks}} a_{s(\mu)} v_{r(\mu)}^{(k)}.
$$
Comparing  coefficients  in the above two formulas one can see that 
\begin{equation}\label{pomocnicza relacja} 
a_v=\sum_{r(e)=v } x_{s(e)} + \sum_{k=1,..., N} \,\, \sum_{\mu \in E^k, r(\mu)=v} a_{s(\mu)}\qquad \textrm{ for }\,\, v\in E^0_{sinks}
\end{equation}
(in particular $a_v=0$ for $v\in  E^0_{sinks}\setminus r(E^1)$), and
\begin{equation}\label{pomocnicza relacja1}
\sum_{\mu \in E^N, r(\mu)=v} a_{s(\mu) }= x_v - \sum_{r(e)=v } x_{s(e)},  \qquad   \textrm{ for }\,\,v\in r (E^N)\setminus  E_{sinks}^0.
\end{equation}
We  define an element of $\Z (E^{0}\setminus E^{0}_{sinks})$ by
$$
b:=\sum_{v\in r(E^N)\setminus E_{sinks}^0} x_v v +\sum_{ k=0,..., N-1} \,\, \sum_{\mu \in E^k\setminus E^k_{sinks}} a_{s(\mu)}  r(\mu). 
$$
Using  \eqref{pomocnicza relacja} and \eqref{pomocnicza relacja1}, in the third equality below,  we obtain 
\begin{align*}
\Delta_E b
&=b- \sum_{v\in r(E^N)} \left(\sum_{r(e)=v } x_{s(e)}\right)v -\sum_{\mu \in E^k, k=1,..., N} a_{s(\mu)} r(\mu)
\\
&=\sum_{v\in r(E^N)\setminus E_{sinks}^0} \left(x_v - \sum_{r(e)=v } x_{s(e)} - \sum_{\mu \in E^N, r(\mu)=v} a_{s(\mu) }\right) v  
\\
& + \sum_{v\in E^0\setminus E^0_{sinks}} a_{v} -\sum_{v\in   E_{sinks}^0} \left(\sum_{r(e)=v } x_{s(e)} +\sum_{k=1,..., N}\,\, \sum_{\mu \in E^k, r(\mu)=v} a_{s(\mu)}\right) v
\\ 
&=0+ \sum_{v\in E^0\setminus E_{sinks}^0} a_{v} v + \sum_{v\in  E_{sinks}^0}  a_{v} v= a.
\end{align*}
\end{proof}
\begin{cor}[cf. Theorem 3.2 in \cite{Raeburn Szymanski}]
    We have isomorphisms
    $$
   K_0(C^*(E))\cong \ker(\Delta_E), \qquad
            K_1(C^*(E))  \cong \coker(\Delta_E). 
    $$
\end{cor}


\begin{thebibliography}{99}
 \bibitem{aee}
 B. Abadie, S. Eilers, R. Exel, \emph{Morita equivalence for crossed products by Hilbert $C^*$-bimodules}.
Trans. Amer. Math. Soc.
 \textbf{350}(8) (1998),  3043--3054. 






 \bibitem{Ant-Bakht-Leb}
 A. B. Antonevich, V. I. Bakhtin, A. V. Lebedev, \emph{Crossed product of $C^*$-algebra by an endomorphism,
  coefficient algebras and transfer operators}.  Math. Sbor. \textbf{202}(9) (2011), 1253--1283.
   


 \bibitem{Ara-Exel-Katsura}
P. Ara, R. Exel, T. Katsura,  \emph{Dynamical systems of type (m,n) and their C*-algebras}. Ergodic Theory Dynam. Systems to appear, arXiv:1109.4093.
   
  

\bibitem{bprs} T. Bates, D. Pask, I. Raeburn, W. Szyma\'nski, \emph{The $C^*$-algebras of row-finite graphs}.
 New York J. Math. \textbf{6} (2000), 307--324.





\bibitem{Bratteli} O. Bratteli,  \emph{Inductive limits of finite dimensional $C^*$-algebras}.
  Trans. Amer. Math. Soc.   \textbf{171} (1972), 195--234.
  
 
  
 \bibitem{BMS} L.G. Brown, J. Mingo, N. Shen, \emph{Quasi-multipliers and embeddings of Hilbert $C^*$-modules}.  Canad. J. Math.  \textbf{71 }(1994), 1150--1174.
  
  
 \bibitem{BRV}  N. Brownlowe,  I. Raeburn, S. Vittadello,  \emph{Exel's crossed product for non-unital C*-algebras}.  Math. Proc. Camb. Phil. Soc. \textbf{149} (2010), 423--444. 
  
\bibitem{Connes} 
A. Connes. \emph{Noncommutative Geometry.} Academic Press, San Diego, CA, 1994. 


 \bibitem{Cuntz1977} J. Cuntz, \emph{Simple {$C\sp*$}-algebras
    generated by isometries}. Comm. Math. Phys. \textbf{57}(2)
    (1977), 173--185.
 
  
\bibitem{CK} J. Cuntz, W. Krieger, \emph{A class of $C^*$-algebras and topological Markov chains}.
Invetiones Math. \textbf{56} (1980), 256--268.  






 \bibitem{exel2} R. Exel \emph{A new look at the crossed-product of a
$C^*$-algebra by an endomorphism}.
Ergodic Theory Dynam. Systems   \textbf{23} (2003),  1733--1750. 

\bibitem{exel-inter}
R. Exel, \emph{Interactions}, J. Funct. Analysis \textbf{244} (2007), 26--62. 

 
\bibitem{exel3} R. Exel, \emph{A new look at the crossed-product of a
$C^*$-algebra by a semigroup of endomorphisms}.
Ergodic Theory Dynam. Systems   \textbf{28} (2008),  749--789.
 
\bibitem{exel4} R. Exel, \emph{Interactions and dynamical systems of type $(n,m)$ - a case study}. 
Preprint, arXiv:1212.5963




\bibitem{EL1}
 R. Exel,  M. Laca,  \emph{Cuntz-Krieger algebras for infinite matrices}. J. Reine Angew. Math., \textbf{512} (1999), 119--172.


\bibitem{EL2}
R. Exel and M. Laca, \emph{The $K$-Theory of Cuntz-Krieger algebras for infinite matrices}. K-Theory, \textbf{19} (2000), 251--268.





\bibitem{exel-renault} R. Exel,  J. Renault,   \emph{Semigroups of local homeomorphisms and interaction groups}.
Ergodic Theory Dynam. Systems   \textbf{27} (2007),  1737--1771.
 
\bibitem{exel_vershik}  R. Exel and A. Vershik, 
\emph{$C^*$-algebras of irreversible dynamical systems}. 
Canadian J. Math. {\bf 58} (2006), 39--63. 



\bibitem{hr}
A. an Huef, I. Raeburn,  \emph{Stacey crossed products associated to Exel systems}.  Integr. Equ. Oper. Theory \textbf{72} (2012), 537--561.


\bibitem{jp} Ja A. Jeong,  Gi Hyun Park,  \emph{Topological entropy for the canonical
completely positive maps on graph $C^*$-algebras}.  Bull. Austral. Math. Soc. \textbf{70}(1) (2004), 101-–116.


\bibitem{KajiWata} T. Kajiwara, Y. Watatani, \emph{Ideals of the core of $C^*$-algebras associated with self-similar maps}. Preprint, arXive:1306.1878


\bibitem{katsura} T. Katsura, \emph{On $C^*$-algebras associated with $C^*$-correspondences}. J. Funct. Anal. \textbf{217}(2) (2004),  366--401.





\bibitem{kwa-trans} B. K. Kwa\'sniewski, \emph{ On transfer operators for $C^*$-dynamical systems}.  Rocky J. Math. \textbf{42}(3) (2012), 919--938.
\bibitem{kwa-logist} B. K. Kwa\'sniewski, \emph{$C^*$-algebras associated with reversible extensions of logistic maps}. Mat. Sb. \textbf{203}(10) (2012), 1448--1489.


 \bibitem{kwa} B. K. Kwa\'sniewski,  \emph{Topological freeness for  Hilbert bimodules}.  To appear in  \emph{Israel J. Math.} DOI: 10.1007/s11856-013-0057-0 


\bibitem{kwa-exel} B. K. Kwa\'{s}niewski,  \emph{Exel's crossed products and crossed products 
by completely positive maps}. Preprint arXiv:1404.4929

\bibitem{kwa-ext} B. K. Kwa\'sniewski, \emph{Extensions of $C^*$-dynamical systems  to systems  with  complete transfer operators}. Preprint, arXive:math/0703800 



\bibitem{kwa-demon} B. K. Kwa\'sniewski, \emph{Crossed product of a $C^*$-algebra by a semigroup of interactions}. To appear in Demonstratio Math. 

 
\bibitem{kum-pask-rae}
A. Kumjian, D. Pask, I. Raeburn,  \emph{Cuntz-Krieger algebras of directed graphs}.
Pacific J. Math. \textbf{184} (1998), 161--174.



\bibitem{kum-pask-rae-ren} A. Kumjian, D. Pask, I. Raeburn, J. Renault,  \emph{Graphs, groupoids and Cuntz-Krieger algebras}.  J. Funct. Anal. \textbf{144} (1997), 505–-541.

\bibitem{Lac-Spiel} M. Laca, J. Spielberg,  \emph{Purely infinite $C^*$-algebras from boundary actions of discrete groups}.  J. Reine. Angew. Math. \textbf{480} (1996), 125–-139.


\bibitem{Donovan} 
D. P. O’Donovan, \emph{Weighted shifts and covariance algebras}. Trans. Amer. Math. Soc.  \textbf{208} (1975), 1–-25.

\bibitem{OlesenPedersen2}
D. Olesen and G. K. Pedersen, \emph{Applications of the Connes spectrum to $C^*$-dynamical
systems II}. J. Funct. Anal. \textbf{36} (1980), 18--32.

\bibitem{OlesenPedersen3}
D. Olesen and G. K. Pedersen, \emph{Applications of the Connes spectrum to $C^*$-dynamical
systems III}. J. Funct. Anal. \textbf{45} (1982), 357--390.


\bibitem{Paschke0} W. Paschke,  \emph{The crossed product of a $C^*$-algebra by an endomorphism}, Proc. Amer. Math. Soc. \textbf{80} (1980), 113--118.


 
\bibitem{Paschke1} 
W.L. Paschke, \emph{$K$-theory for actions of the circle group on $C^*$-algebras}. J. Operator Theory \textbf{6}
(1981), 125--133. 


 \bibitem{Pedersen} G. K. Pedersen. \emph{$C^*$-algebras and their automorphism groups}, Academic Press, London, 1979.


 \bibitem{p} M. V. Pimsner,  \emph{A class of $C^*$-algebras generalizing both Cuntz-Krieger algebras and crossed products by $\Z$}. Fields Institute Communications {\bf 12} (1997), 189--212.


\bibitem{Raeburn}
 I. Raeburn, \emph{Graph Algebras}. BMS Regional Conference Series in Mathematics, vol. 103, Amer. Math. Soc., Providence, 2005.
 
 
 \bibitem{Raeburn Szymanski}
I. Raeburn, W. Szyma\'nski, \emph{Cuntz-Krieger algebras of infinite graphs and matrices}. Trans. Amer. Math. Soc. \textbf{356} (2004), 39--59.

\bibitem{morita}
I. Raeburn, D. P. Williams,  \emph{Morita equivalence and continuous-trace $C^*$-algebras}.
  Amer. Math. Soc., Providence, 1998.

\bibitem{Rordam}
 M. R\o rdam,  \emph{Classification of certain infinite simple $C^*$-algebras}. J. Funct. Anal. \textbf{131} (1995), 415–-458.

\bibitem{Szwajcar} J. Schweizer, \emph{Dilations of $C^*$-correspondences and the simplicity of Cuntz-Pimsner algebras}. J. Funct. Anal. \textbf{180} (2001), 404--425.
\end{thebibliography}
\end{document}